\documentclass[preprint]{elsarticle}
\usepackage [english]{babel}
\usepackage[utf8]{inputenc}
\usepackage{amssymb}
\usepackage{amsbsy}
\usepackage{amsmath}
\usepackage{amsfonts}
\usepackage{booktabs}
\usepackage{graphicx}
\usepackage{url}
\usepackage{algorithm}
\usepackage{parskip}
\usepackage{makecell}
\usepackage{color}
\usepackage{xifthen}
\usepackage{hyperref}
\journal{arXiv}
\newcommand{\ol}{\overline}
\newcommand{\ul}{\underline}

\newcommand{\GrU}{\textnormal{x}}
\newcommand{\Gr}{\textnormal{\textbf{x}}}
\newcommand{\xx}{\textnormal{\textbf{x}}}
\newtheorem{theorem}{Theorem}
\newtheorem{lemma}{Lemma}
\newtheorem{corollary}{Corollary}
\newtheorem{definition}{Definition}
\newtheorem{assumption}{Assumption}
\newdefinition{remark}{Remark}
\newproof{proof}{Proof}

\numberwithin{equation}{section}
\numberwithin{theorem}{section}
\numberwithin{lemma}{section}
\numberwithin{corollary}{section}
\numberwithin{definition}{section}
\numberwithin{assumption}{section}
\numberwithin{remark}{section}

\begin{document}

\begin{frontmatter}

\title{Robust approximation error estimates 
and multigrid solvers
for isogeometric multi-patch discretizations}

\author[ricam]{Stefan Takacs}
\ead{stefan.takacs@ricam.oeaw.ac.at}

\address[ricam]{Johann Radon Institute for Computational and Applied Mathematics (RICAM),\\
Austrian Academy of Sciences}

\begin{abstract}
In recent publications, the author and his coworkers have shown robust
approximation error estimates for B-splines of maximum smoothness and have
proposed multigrid methods based on them. These methods allow to solve the linear system
arizing from the discretization of a partial differential equation in Isogeometric
Analysis in a single-patch setting with convergence rates that are provably robust 
both in the grid size and the spline degree. In real-world problems, the computational
domain cannot be nicely represented by just one patch.
In computer aided design, such domains are typically represented as a union of multiple patches.
In the present paper, we extend the approximation
error estimates and the multigrid solver to this multi-patch case.
\end{abstract}


\end{frontmatter}

\section{Introduction}

The key idea of Isogeometric Analysis (IgA),~\cite{Hughes:2005}, is to unite the world of
computer aided design (CAD) and the world of finite element (FEM) simulation. 
Spline spaces, such as spaces spanned by tensor product B-splines or NURBS,
are typically used for geometry representation in standard CAD systems. In classical
IgA, both the computational domain and the solution of the partial differential equation (PDE)
are represented by spline functions.

More complicated domains cannot be represented by just one such (tensor-product)
spline function. Instead, the whole domain is decomposed into
subdomains, in IgA typically called patches, where each of them is
represented by its own geometry function. This is  called the
\emph{multi-patch case}, in contrast to the \emph{single-patch case}.

Concerning the approximation error, in early IgA literature, only its
dependence on the grid size has been studied, cf.~\cite{Hughes:2005,Bazilevs:2006}. In recent
publications~\cite{daVeiga:2011,Takacs:Takacs:2015,Floater:Sande:2017} also the
dependence on the spline degree has been investigated.
These error estimates are restricted to the single-patch case.
We will extend the results from~\cite{Takacs:Takacs:2015} on approximation errors
for B-splines of maximum smoothness to the multi-patch case.

As a next step, the linear system resulting from the isogeometric discretization
of the PDE has to be solved.
Several solvers have been proposed for the multi-patch case, typically established
solution strategies known from the finite element literature, including
direct solvers \cite{CollierEtAl:2012} or non-overlapping and overlapping domain
decomposition methods
\cite{DaVeigaEtAl:2012,DaVeigaEtAl:2013,DaVeigaEtAl:2014}, FETI-like approaches
(called IETI in the IgA context)~\cite{KleissEtAl:2012}.
The solution of local subproblems in such domain
decomposition methods is done with general direct solvers, 
fast direct solvers exploiting the tensor product structure, cf.~\cite{Sangalli:2016},
or again iterative solvers, like multigrid or multilevel methods, 
cf.~\cite{HLT:2017} for multigrid methods in the framework of a IETI solver.

To apply multigrid methods directly to the system
arizing from a multi-patch discretization, is an appealing alternative.
If standard smoothers known from finite elements (Jacobi,
Gauss Seidel) are used, the extension of the multigrid methods to multi-patch IgA discretizations
is straight-forward. However, it is well known that their convergence rates
deteriorate dramatically if $p$ is increased, cf.~\cite{GahalautEtAl:2013,HofreitherZulehner:2014c,HTZ:2016}.

A robust and efficient multigrid solver for the single-patch case
was presented in~\cite{HT:2016}; alternatives include~\cite{Donatelli:2014a,HTZ:2016}.
Based on a robust inverse inequality and a robust approximation error estimate in a large 
subspace of the whole spline space (from~\cite{Takacs:Takacs:2015}), it was shown
that mass matrices can be used as robust smoothers in this large subspace. For
the other subspaces, particular
smoothers have been proposed, which can capture the outlier frequencies
on the one hand and which still have tensor product structure on the other
hand. The overall smoother is then obtained by combining them by an additive
Schwarz type approach.

That multigrid smoother relies on the tensor-product structure of the mass matrix
and is, therefore, restricted to the single-patch case. We will set up instances of
that smoother for each patch and will combine them in an additive Schwarz
type way to obtain a multi-patch multigrid smoother. This smoother will be used
in a standard multigrid framework living on the whole multi-patch domain.
We will discuss the convergence
rates of the multigrid solver and its overall computational complexity.

Multigrid methods are typically known as optimal methods, which means that their
overall computational complexity grows linearly with the number of unknowns.
If also the dependence in the spline
degree is of interest, the best we can expect is that the multigrid method is
not more expensive than the computation of the residual, which requires
the multiplication with the stiffness matrix. In two dimensions, the stiffness matrix
has $\mathcal{O}(N p^2)$ non-zero entries, where $N$ is the number of unknowns,
$p$ is the spline degree, and $\mathcal{O}(\cdot)$ is the Landau notation. So, we call
the multigrid method \emph{optimal} if we can show that its 
overall complexity is not more than $\mathcal{O}(N p^d)$.

The remainder of the paper is organized as follows. First, the model problem and
the discretization are discussed in Section~\ref{sec:prelim}. Then, in Section~\ref{sec:approx},
a robust approximation error estimate for the multi-patch domain is given. These
results are used in Section~\ref{sec:mg} to set up a multigrid method for the multi-patch
domain. In Section~\ref{sec:num}, we give numerical experiments for the multigrid  method
and in Section~\ref{sec:conclusions}, we draw conclusions.

\section{Preliminaries}\label{sec:prelim}
 
In this paper, we consider the following \emph{Poisson model problem}.
For a given function $f$, we are interested in the function $u$ solving
\[
		- \Delta u = f \quad\mbox{ in }\Omega, \qquad
		u = 0 \quad\mbox{ on }\partial \Omega,
\]
where $\Omega \subset \mathbb{R}^2$ is an open, bounded and
simply connected Lipschitz domain with boundary $\partial \Omega$.
The standard weak form of the model problem reads as follows. Given $f\in L_2(\Omega)$,
find $u\in H^1_0(\Omega)$ such that
\begin{equation} \label{eq:model}
			(\nabla u,\nabla v)_{L_2(\Omega)} = (f,v)_{L_2(\Omega)}
\qquad \mbox{for all $v \in H^1_0(\Omega)$.}
\end{equation}
Here and in what follows, $L_2(\Omega)$, $H^1(\Omega)$, $H^2(\Omega)$ and $H^1_0(\Omega)$
are the standard Lebesgue and Sobolev spaces with standard scalar products
$(\cdot,\cdot)_{L_2(\Omega)}$, $(\cdot,\cdot)_{H^1(\Omega)}:=(\nabla\cdot,\nabla\cdot)_{L_2(\Omega)}$, 
norms $\|\cdot\|_{L_2(\Omega)}$, $\|\cdot\|_{H^1(\Omega)}$, $\|\cdot\|_{H^2(\Omega)}$,
and seminorm $|\cdot|_{H^1(\Omega)}$.

This problem is solved with a standard \emph{fully matching multi-patch isogeometric discretization}.
For sake of completeness and to introduce a notation, we give the details. 
For simplicity, we restrict ourselves to the two-dimensional case.

Assume that the domain $\Omega\subset\mathbb{R}^2$ consists of $K$ patches, denoted by $\Omega_k$
for $k=1,\ldots,K$ such that the domain $\Omega$ is covered by non-overlapping patches, i.e.,
\begin{equation} \label{eq:matching}
	\ol{\Omega} = \bigcup_{k=1}^K \ol{\Omega_k}
	\quad \mbox{and} \quad
	\Omega_k\cap\Omega_l=\emptyset \mbox{ for any }k\not=l,
\end{equation}
where for any domain $T\subset \mathbb{R}^2$, the symbol $\ol{T}$ denotes its closure.
Each of those patches is represented by a bijective geometry function
\[
		G_k :\widehat{\Omega}:=(0,1)^2 \rightarrow \Omega_k := G_k (\widehat{\Omega})\subset \mathbb{R}^2,
\]
which can be continuously extended to the closure of $\widehat{\Omega}$.

Analogously to~\cite{HT:2016}, we assume that the geometry function is
sufficiently smooth such that the following assumption holds.
\begin{assumption}\label{ass:geoequiv}
	There is a constant $C_G>0$ such that geometry functions $G_k$ satisfy
	\begin{align*}
		C_G^{-1} \|v\|_{L_2(\widehat{\Omega})} \le \|v\circ G_k^{-1}\|_{L_2(\Omega_k)}  \le C_G \|v\|_{L_2(\widehat{\Omega})}
		&\;\; \mbox{ for all } v \in L_2(\widehat{\Omega})\\
		C_G^{-1} \|v\|_{H^r(\widehat{\Omega})} \le \|v\circ G_k^{-1}\|_{H^r(\Omega_k)}  \le C_G \|v\|_{H^r(\widehat{\Omega})}
		&\;\; \mbox{ for all } v \in H^r(\widehat{\Omega}), \, r\in\{1,2\}.
	\end{align*}
\end{assumption}
As the dependence on the geometry function is not in the focus of this paper,
unspecified constants might depend on $C_G$.

For any patch $\Omega_k$, we denote 
by $\mathbb K_k :=\{G_k( (0,1)^2 )\} = \{\Omega_k\}$ its interior, by
\begin{align*}
	\mathbb E_k & :=
	\left\{ G_k(\Gamma) \;:\;
			\begin{array}{c} \Gamma \in \{ \{0\}\times (0,1), \{1\}\times (0,1), (0,1)\times\{0\}, (0,1)\times\{1\}\} \\
					\mbox{ such that }  G_k(\Gamma) \not \subset \partial \Omega \end{array} \right\}
\end{align*}
its edges and by
\begin{align*}
	\mathbb V_k & :=
	\{ G_k(\{(\alpha,\beta)\}) \;:\; (\alpha,\beta)\in\{0,1\}^2\mbox{ such that } G_k(\{(\alpha,\beta)\}) \not\subset \partial\Omega\} 
\end{align*}
its vertices, where in both cases edges and vertices located on the (Dirichlet) boundary of $\Omega$ are excluded.
 $\mathbb T_k:= \mathbb K_k \cup \mathbb E_k \cup \mathbb V_k$ denotes all pieces of $\Omega_k$.
The following assumption excludes hanging vertices.
\begin{assumption}\label{ass:no:hanging}
	The intersection of $\ol{\Omega_k}$ and $\ol{\Omega_l}$ for $k\not=l$ is either (a)
	empty, (b) one common vertex or (c) the union of one common edge and two common vertices.
\end{assumption}
We define the set of all interiors $\mathbb K := \bigcup_{k=1}^K \mathbb K_k$,
edges $\mathbb E := \bigcup_{k=1}^K \mathbb E_k$,
 vertices $\mathbb V := \bigcup_{k=1}^K \mathbb V_k$,
pieces $\mathbb T := \bigcup_{k=1}^K \mathbb T_k = \mathbb K \cup \mathbb E \cup \mathbb V $
and observe that using Assumption~\ref{ass:no:hanging}, we obtain that the pieces form a partition of $\Omega$:
\[
	\Omega = \bigcup_{T \in \mathbb T} T
	\quad \mbox{and} \quad
	S\cap T=\emptyset \mbox{ for any }S,T\in \mathbb T,\;S\not=T.
\]
Finally, we assume that the number of neighbors of each patch is uniformly bounded.
\begin{assumption}\label{ass:neigbour}
	Assume that none of the vertices $T \in \mathbb V$ contributes to more than $C_N$
	patches, i.e., $| \{ k\;:\; T \subset \ol{\Omega_k} \} |\le C_N$.
\end{assumption}

Now, having a representation of the domain, we introduce the isogeometric
function space.

For the \emph{univariate case}, the
space of spline functions of degree $p\in \mathbb{N}:=\{1,2,\ldots\}$
and size $h=m^{-1}$ with $m \in  \mathbb{N}$ is given by 
	\begin{equation*}
		S_{p,h} := \left\{ v \in C^{p-1}(0,1): \; v |_{((j-1)h,j\, h]} \in \mathbb{P}^p \mbox{ for all } j=1,\ldots,m \right\},
	\end{equation*}
where $\mathbb{P}^p$ is the space of polynomials of degree $p$ and $C^{p-1}(0,1)$
is the space of all $p-1$ times continuously differentiable functions.

We denote the standard basis for $S_{p,h}$, as introduced by the Cox-de Boor formula,
cf.~\cite{DeBoor:1972}, by
$\Phi_{p,h}:=(\widehat{B}_{p,h}^{(i)})_{i=1}^n$, where~$n=m+p$
is the dimension of the spline space. Note that only the first 
basis function
\[
\widehat{B}_{p,h}^{(1)} = \max\{0,(1-x/h)^p\}
\]
contributes to the left boundary. Analogously, only the last
basis function contributes to the right boundary.
We assign  corresponding Greville
points $0=\widehat{\GrU}_{p,h}^{(1)} < \widehat{\GrU}_{p,h}^{(2)} 
<\cdots < \widehat{\GrU}_{p,h}^{(n)} =1$ to the basis functions.

On the \emph{parameter domain} $\widehat{\Omega}$, we introduce for each patch tensor-product B-spline functions
\begin{equation}\label{eq:vhatdef}
		\widehat{V}_k := S_{p,h} \otimes S_{p,h}
\end{equation}
with basis $\widehat{\Phi}_{k} := ( \widehat{B}^{(i)}_k )_{i=1}^{n^2}$, where the basis functions and the Greville
points are given by
\begin{equation}\label{eq:vhatdef2}
			\widehat{B}^{(i+n\,(j-1))}_k(x,y)=\widehat{B}_{p,h}^{(i)}(x) \widehat{B}_{p,h}^{(j)}(y)
\quad\mbox{and}\quad \widehat{\Gr}_{k}^{(i+n\,(j-1))} = (\widehat{\GrU}_{p,h}^{(i)},\widehat{\GrU}_{p,h}^{(j)}).
\end{equation}
For sake of simplicity of the notation, we do not indicate the dependence of $p$, $h$,
or $m$ on the patch index $k$ and the spacial direction.

On the \emph{physical domain} $\Omega_k$, we define the ansatz functions using the pull-back principle
\begin{equation}\label{eq:vkdef}
	V_k := \{ u \in H^1(\Omega_k) \; :\; u\circ G_k \in \widehat{V}_k\}
\end{equation}
and obtain
the basis by $\Phi_{k} := (B^{(i)}_k )_{i=1}^{n^2}$ and $B^{(i)}_k := \widehat{B}^{(i)}_k
\circ G_k^{-1}$ and the Greville points by $\Gr_k^{(i)} = G_k(\widehat{\Gr}_k^{(i)})$. 

We require that the function spaces are fully matching on the interfaces.
\begin{assumption}\label{ass:fully}
	For any $T \in \mathbb E$ being a common edge of the patches
	$\Omega_k$ and $\Omega_l$ (i.e., $T\subset \partial \Omega_k \cap \partial \Omega_l$),
	we assume that the basis functions of the two patches and the corresponding Greville points
	match, i.e., for all $i$ there is some $j$ such that
	\begin{equation}\label{eq:fully:matching2}
		B_k^{(i)}|_{T} = B_l^{(j)}|_{T}
		\qquad \mbox{and} \qquad
		\Gr_k^{(i)} = \Gr_l^{(j)} 
	\end{equation}
	holds, where $\cdot|_{T}$ is the trace operator.
\end{assumption}

The  \emph{multi-patch function space} $V_h$  is given by
\[
	V_h :=
	\{
	u \in H^1_0(\Omega)
		\; : \;
		u|_{\Omega_k} \in V_k \mbox{ for } k=1,\ldots,K
	\}.
\]
For this space, we introduce a set of global basis functions by 
\begin{equation}\label{eq:phi}
		\Phi  := \{ \phi_{\Gr_k^{(i)}} \;:\; k\in\{1,\ldots,K\},\,i\in\{1,\ldots,n^2\} \mbox{ such that } \Gr_k^{(i)} \in \Omega \},
\end{equation}
where the basis functions $\phi_{\xx} \in V_h$ are such that
\begin{align*}
	 \phi_{\xx} |_{ \Omega_k} = \Big\{
			\begin{array}{ll}
				B_k^{(i)} & \mbox{where } i \mbox{ is such that } \Gr_k^{(i)} = {\xx} \mbox{ if } {\xx}\in\ol{\Omega_k} \\
				0	& \mbox{if } {\xx}\not\in\ol{\Omega_k} 
			\end{array}
	\;\; \mbox{for all } k=1,\ldots,K.
\end{align*}
Note that the condition $\Gr_k^{(i)} \in \Omega$ in~\eqref{eq:phi} excludes the basis functions assigned
to the boundary~$\partial\Omega$ and guarantees that the homogenous Dirichlet boundary conditions are satisfied.
By  numbering  the basis functions in $\Phi$ arbitrarily, we obtain $\Phi = \{\phi_i \; :\; i=1,\ldots,N\}$
and a basis $\Phi_h:=(\phi_i)_{i=1}^N$ of $V_h$. 

Note that by construction only the basis
functions whose Greville points are located on an edge (or the corresponding vertices) contribute to that edge
and only the basis function whose Greville point is located on an vertex contributes to that vertex. So, for any
piece $T\in \mathbb T$, we collect the corresponding  functions:
\[
	\Phi^{(T)} := \{  \phi_{\xx}  \in \Phi  \;:\; {\xx} \in T \}.
\]

We use a \emph{standard Galerkin scheme} to discretize~\eqref{eq:model}
and obtain the following discretized problem:
Find $u_h\in V_h$ such that
\begin{equation} \label{eq:model:discr}
			(\nabla u_h,\nabla v_h)_{L_2(\Omega)} = (f,v_h)_{L_2(\Omega)}
\qquad \mbox{for all $v_h \in V_h$.}
\end{equation}
Using the basis $\Phi_h$, we obtain a standard matrix-vector problem: Find $\ul{u}_h \in \mathbb{R}^N$ such
that
\begin{equation} \label{eq:linear:system}
			A_h \ul{u}_h = \ul{f}_h.
\end{equation}
Here and in what follows,
$
	A_h := [ (\nabla \phi_i, \nabla \phi_j)_{L_2(\Omega)} ]_{i,j=1}^N
$
is the standard stiffness matrix,
$
	M_h := [ ( \phi_i,  \phi_j)_{L_2(\Omega)} ]_{i,j=1}^N
$
is the standard mass matrix,
$\ul{u}_h=[u_i]_{i=1}^N$ is the coefficient vector representing $u_h$ with respect
to the basis $\Phi_h$, i.e., $u_h=\sum_{i=1}^N u_i \phi_i$, and $\ul{f}_h =  [ (f, \phi_i)_{L_2(\Omega)} ]_{i=1}^N$ is the coefficient
vector obtained by testing the right-hand-side functional with
the basis functions.

Before we proceed, we introduce a convenient notation.
\begin{definition}
	Any generic constant $c>0$ used within this paper is understood to be 
	independent of the grid size $h$, the spline degree $p$ and the number of patches $K$,
	but it might depend on the shape of $\Omega$, and on the constants $C_G$ and $C_N$.

	We use the notation $a\lesssim b$ if there is a generic constant $c$ such that $a\le c b$ and
	the notation $a \eqsim b$ if $a\lesssim b$ and $b\lesssim a$.

	For symmetric positive definite matrices $A$ and $B$, we write
	\[
		A \le B \qquad \mbox{if} \qquad 
		\ul{u}^{\top} A \ul{u} \le\ul{u}^{\top} B \ul{u} \quad \mbox{ for all vectors } \ul{u}.
	\]
	The notations $A\lesssim B$ and $A \eqsim B$ are defined analogously.
\end{definition}

Following the standard line of arguments, the Lax Milgram lemma and Friedrichs' inequality indicate 
existence and uniqueness of a solution $u\in H^1_0(\Omega)$ for the continuous problem~\eqref{eq:model}
and of a solution $u_h \in V_h$ for the discrete problem~\eqref{eq:model:discr}. Cea's lemma yields
\[
	\| u-u_h \|_{H^1(\Omega)}^2 \lesssim \inf_{v_h \in V_h} \| u-v_h \|_{H^1(\Omega)}^2,
\]
i.e., that the discretization error is bounded by a constant times the approximation error, which motivates to discuss approximation error estimates in the next section.

\section{Robust multi-patch spline approximation}\label{sec:approx}

In this paper, we extend the robust $L_2-H^1$ and $H^1-H^2$-approximation
error estimates from~\cite{Takacs:Takacs:2015} to
 multi-patch domains. For this purpose, we introduce  a projector into the spline space
which is interpolatory on the boundary. This is first done in the one dimensional case
(Section~\ref{sec:projector:1}) and then extended to the two-dimensional case
(Section~\ref{sec:projector:2}). Based on that projector, a projector for multi-patch
domains is introduced (Section~\ref{sec:projector:3}). All of the projectors satisfy
the usual $p$-robust approximation error estimates.

\subsection{The one dimensional case}\label{sec:projector:1}

First, we define an augmented $H^1$-scalar product.
\begin{definition}
	The scalar product $(\cdot,\cdot)_{H^1_D(0,1)}$ is given by
	\begin{equation}\label{eq:h1Ddef}
		(u,v)_{H^1_D(\Omega)} := (u,v)_{H^1(0,1)} + u(0)v(0).
	\end{equation}
\end{definition}
As the scalar product does not have a kernel, it induces a norm
$\|u\|_{H^1_D(0,1)}^2:= (u,u)_{H^1_D(0,1)}$ and the following definition 
introduces an unique projector.
\begin{definition}
	The projector $\Pi_{p,h}:\,H^1(0,1)\rightarrow S_ {p,h}$ is the $H^1_D$-orthogonal
	projection, i.e., for any $u\in H^1(0,1)$, the spline $u_{p,h}:=\Pi_{p,h} u$ satisfies
	\begin{equation}\label{eq:h1Dorth}
		(u-u_{p,h},v_{p,h})_{H^1_D(0,1)} = 0 \quad \mbox{ for all } v_{p,h} \in S_{p,h}.
	\end{equation}
\end{definition}
We observe that the original function and the spline function coincide on both boundary
points and that they are orthogonal in $(\cdot,\cdot)_{H^1(0,1)}$.
\begin{lemma}\label{lem:bc}
	For all $u\in H^1(0,1)$, the spline $u_{p,h}:=\Pi_{p,h} u$ satisfies
	\begin{equation}\label{eq:bc}
		u(0) = u_{p,h}(0),
		\qquad 
		u(1) = u_{p,h}(1)
	\end{equation}
	and
	\begin{equation}\label{eq:H1orth}
		(u-u_{p,h},v_{p,h})_{H^1(0,1)} = 0 \quad \mbox{ for all } v_{p,h} \in S_{p,h}.
	\end{equation}
\end{lemma}
\begin{proof}
	The first statement is obtained by plugging $v(x):=1$ into~\eqref{eq:h1Dorth}.

	For the second statement, we plug $v(x):=x$ into~\eqref{eq:h1Dorth} and obtain
	\begin{align*}
		0& = ( u - u_{p,h}, v)_{H^1_D(0,1)}
		=u(0) - u_{p,h}(0)
			+ \int_{0}^1 u'(x) - u'_{p,h}(x) \mbox{ d}x\\
		&=u(1) - u_{p,h}(1).
	\end{align*}	
	
	For the last statement~\eqref{eq:H1orth},
	observe that~\eqref{eq:h1Dorth} together with~\eqref{eq:h1Ddef} yields
	\[
		(u-u_{p,h},v_{p,h})_{H^1(0,1)} + ( u(0) - u_{p,h}(0) ) v_{p,h}(0) =0
	\]
	for all $v_{p,h} \in S_{p,h}$, which shows together with~\eqref{eq:bc} the desired result.
	\qed
\end{proof}

From~\eqref{eq:H1orth}, we immediately obtain the $H^1$-stability:
\begin{equation}\label{eq:H1stab}
	|\Pi_{p,h} u|_{H^1(0,1)}\le |u|_{H^1(0,1)}.
\end{equation}
Moreover, we obtain the usual approximation error estimates.

\begin{theorem}\label{thrm:approx1}
    For all $u\in H^2(0,1)$, grid sizes $h$
    and spline degrees $p\in\mathbb{N}$, we obtain
    \begin{equation}\label{eq:thrm:approx1}
			    	|u-\Pi_{p,h} u |_{H^1(0,1)} \le \sqrt{2}\; h |u|_{H^2(0,1)}.
    \end{equation}
\end{theorem}
\begin{proof}
	We have $|u-\Pi_{p,h} u |_{H^1(0,1)} = \inf_{u_{p,h}\in S_{p,h}} 
	|u-u_{p,h}|_{H^1(0,1)}$ because $\Pi_{p,h}$ minimizes the $H^1$-seminorm.
	For the case $h<p^{-1}$, the estimate directly follows from~\cite[Theorem~7.3]{Takacs:Takacs:2015}.
	For  $h>p^{-1}$, we use that the space of global polynomials is a subspace of the spline
	space. So, \cite[Theorem~3.17]{Schwab:1998} yields (for $M=1$, $\Omega=\Omega_1=(0,1)$, $k_1=s_1=1$)
	$|u-\Pi_{p,h} u |_{H^1(0,1)} \le2^{-1} (p(p+1))^{-1/2} |u|_{H^2(0,1)}$. Using
	$p^{-1} < h$, we obtain also for this case the desired result.
\qed
\end{proof}

\begin{theorem}\label{thrm:approx1a}
    For all $u\in H^1(0,1)$, grid sizes $h$
    and spline degrees $p\in\mathbb{N}$, we obtain
    \begin{equation}\label{eq:thrm:approx1a}
			    	\|u-\Pi_{p,h} u \|_{L_2(0,1)} \le \sqrt{2}\; h |u|_{H^1(0,1)}.
    \end{equation}
\end{theorem}
\begin{proof}
	This estimate is shown by a classical Aubin Nitsche duality trick.
	Let $v\in H^2(0,1)$ such that $v(0)=v(1)=0$ and $-v''=u-\Pi_{p,h}u$.	
	Then we obtain using integration by parts (the boundary terms vanish
	due to Lemma~\ref{lem:bc})  that
	\begin{align*}
		& \|u-\Pi_{p,h} u \|_{L_2(0,1)}
			 =  \frac{ (u-\Pi_{p,h} u,u-\Pi_{p,h} u)_{L_2(0,1)} }{ \|u-\Pi_{p,h} u\|_{L_2(0,1)} } 
			 =  \frac{ -(u-\Pi_{p,h} u,v'')_{L_2(0,1)} }{ \|v''\|_{L_2(0,1)} } \\
			& \quad =  \frac{ (u-\Pi_{p,h} u,v)_{H^1(0,1)} }{ |v|_{H^2(0,1)} } 
			 \le \sup_{w \in H^2(0,1)}  \frac{ (u-\Pi_{p,h} u,w)_{H^1(0,1)} }{ |w|_{H^2(0,1)} }. 
	\end{align*}
	Using Theorem~\ref{thrm:approx1}, we obtain further
	\begin{align*}
		\|u-\Pi_{p,h} u \|_{L_2(0,1)}
			& \le \sqrt{2}\; h \sup_{w \in H^2(0,1)}  \frac{ (u-\Pi_{p,h} u,w)_{H^1(0,1)} }{ |w-\Pi_{p,h} w |_{H^1(0,1)} } .
	\end{align*}
	With the orthogonality relation~\eqref{eq:H1orth}, the Cauchy-Schwarz inequality, and the
	stability estimate~\eqref{eq:H1stab}, we finally conclude
	\begin{align*}
		\|u-\Pi_{p,h} u \|_{L_2(0,1)}
			& \le \sqrt{2}\; h \sup_{w \in H^2(0,1)}  \frac{ (u-\Pi_{p,h} u,w-\Pi_{p,h} w)_{H^1(0,1)} }{ |w-\Pi_{p,h} w |_{H^1(0,1)} }\\
			&  \le \sqrt{2}\; h |u-\Pi_{p,h}u|_{H^1(0,1)} 
			\le  \sqrt{2}\; h |u|_{H^1(0,1)}.
	\end{align*}
	\vspace{-4em}
	
	\mbox{}\qed
\end{proof}

The projector can be represented by a dual basis.
\begin{lemma}\label{lem:dualbasis}
	For all grid sizes $h$ and spline degrees $p\in\mathbb{N}$, there are dual basis functions
	$\lambda_{p,h}^{(i)}\in S_{p,h}$ for $i=1,\ldots,n$ such that
	\[
		\Pi_{p,h} u = \sum_{i=1}^n (u,\lambda_{p,h}^{(i)})_{H^1_D(0,1)} \widehat{B}_{p,h}^{(i)}
			\qquad \mbox{for all } u\in H^1(0,1).
	\]
\end{lemma}
\begin{proof}
	Let $u\in H^1(0,1)$ be arbitrary but fixed. As $ (\widehat{B}_{p,h}^{(i)})_{i=1}^n$ is a basis
	of $S_{p,h}$, we can expand
	$
		\Pi_{p,h} u = \sum_{i=1}^n u_i \widehat{B}_{p,h}^{(i)}.
	$
	By plugging this into~\eqref{eq:H1orth}, we obtain 
	\begin{align*}
		0 
		&= (u-\Pi_{p,h} u, \widehat{B}_{p,h}^{(j)} )_{H^1_D(0,1)} \
		= (u, \widehat{B}_{p,h}^{(j)} )_{H^1_D(0,1)} -
		\sum_{i=1}^n u_i \underbrace{( \widehat{B}_{p,h}^{(i)}, \widehat{B}_{p,h}^{(j)} )_{H^1_D(0,1)}}_{\displaystyle a_{i,j}:=},
	\end{align*}
	for  $j=1,\ldots,n$. As the ${H^1_D(\Omega)}$-scalar product 
	induces a norm (and not only a seminorm), the stiffness matrix $[a_{i,j}]_{i,j=1}^n$ is non-singular.
	So, there is an inverse matrix $[w_{i,j}]_{i,j=1}^n$ and we obtain
	\begin{align*}
		u_i = \sum_{j=1}^n w_{i,j} (u, \widehat{B}_{p,h}^{(j)} )_{H^1_D(0,1)}
		=   \Big(u,\underbrace{ \sum_{j=1}^n w_{i,j} \widehat{B}_{p,h}^{(j)} }_{\displaystyle \lambda_{p,h}^{(i)}:=} \Big)_{H^1_D(0,1)},
	\end{align*}
	\vspace{-3em}

	which finishes the proof.\qed
\end{proof}

\subsection{The two-dimensional case}\label{sec:projector:2}

For the two-dimensional case on the parameter domain 
$\widehat{\Omega}=(0,1)^2$, we define
the projector $\widehat{\Pi}_k : H^2(\widehat{\Omega})\rightarrow \widehat{V}_k$
using the idea of tensor-product projection.
First, we define the following two projectors on $u\in H^2(\widehat{\Omega})$:
\begin{align*}
			&(\Pi_{p,h}^x u)(\cdot,y) := \Pi_{p,h} u(\cdot,y) \quad \mbox{ for all } y \in (0,1), \\
		    &(\Pi_{p,h}^y u)(x,\cdot) := \Pi_{p,h} u(x,\cdot) \quad \mbox{ for all } x \in (0,1),
\end{align*}
and observe that these operators commute.

\begin{lemma}\label{lem:proj2d}
	We have $\Pi_{p,h}^x \Pi_{p,h}^y =  \Pi_{p,h}^y \Pi_{p,h}^x$.
\end{lemma}
\begin{proof}
	Let $\partial_\xi$, $\partial_\eta$ and $\partial_{\xi\eta}$ be the corresponding partial derivatives.
	Lemma~\ref{lem:dualbasis} guarantees the existence of a dual bases. So, 
  \[
		\Pi_{p,h}^x u(x,y) = \sum_{i=1}^n \left( 
			\int_0^1 \partial_\xi u(\xi,y) \partial_\xi \lambda_{p,h}^{(i)}(\xi) \,\mbox{d}\xi
			+ u(0,y)\lambda_{p,h}^{(i)}(0) \right)
			\widehat{B}_{p,h}^{(i)}(x),
	\]
  and straight forward computations yield
  \begin{align*}
		&\Pi_{p,h}^y \Pi_{p,h}^x u(x,y) \\
		&= \sum_{i=1}^n \sum_{j=1}^n \Big( \int_0^1  
		   \int_0^1 \partial_{\xi\eta} u(\xi,\eta)\, \partial_\xi \lambda_{p,h}^{(i)}(\xi) \,\partial_\eta  \lambda_{p,h}^{(j)}(\eta)\,  \mbox{d}\xi \, 
			 \, \mbox{d}\eta\\
		& \qquad + 
		   \int_0^1 \partial_\xi u(\xi,0)\, \partial_\xi \lambda_{p,h}^{(i)}(\xi) \, \lambda_{p,h}^{(j)}(0)\,  \mbox{d}\xi \,  
			+ 
		   \int_0^1 \partial_\eta u(0,\eta)\, \lambda_{p,h}^{(j)}(0) \, \partial_\eta  \lambda_{p,h}^{(i)}(\eta)\,  \mbox{d}\eta \,  \\
			&\qquad
			+ 
		  u(0,0)\, \lambda_{p,h}^{(i)}(0) \,   \lambda_{p,h}^{(j)}(0)
			\Big) \widehat{B}_{p,h}^{(i)}(x) \widehat{B}_{p,h}^{(j)}(y).
	\end{align*}
	Observe that this term is symmetric in $x$ and $y$. So $\Pi_{p,h}^y \Pi_{p,h}^x =
	\Pi_{p,h}^x \Pi_{p,h}^y$.
\qed\end{proof}

As $\Pi_{p,h}^x \Pi_{p,h}^y =  \Pi_{p,h}^y \Pi_{p,h}^x$, the projector
 \begin{equation}\label{eq:hat:pi:k:def}
	\widehat{\Pi}_k :=\Pi_{p,h}^x \Pi_{p,h}^y
\end{equation}
maps into $\widehat{V}_k$, the intersection of the image spaces of these two projectors.

\begin{theorem}\label{thrm:approx2}
    For all $u\in H^2(\widehat{\Omega})$, grid sizes $h$ 
    and spline degrees $p\in\mathbb{N}$, we obtain
    \begin{equation}\label{eq:thrm:approx2}
			    	|u-\widehat{\Pi}_{k} u |_{H^1(\widehat{\Omega})} \le 2 \;h |u|_{H^2(\widehat{\Omega})}.
    \end{equation}
\end{theorem}
\begin{proof}
	First we show
    \begin{equation*}
			 \|\partial_x(u-\widehat\Pi_{k} u) \|_{L_2(\widehat{\Omega})}^2 
			 			\le 2 h ( \|\partial_{xx} u\|_{L_2(\widehat{\Omega})}^2
			 			+ \|\partial_{xy} u\|_{L_2(\widehat{\Omega})}^2),
    \end{equation*}	
    where $\partial_x$, $\partial_{xx}$ and $\partial_{xy}$ are the corresponding
    partial derivatives.

	Using $\widehat\Pi_k = \Pi_{p,h}^x\Pi_{p,h}^y$, the triangle inequality, the $H^1$-stability of
	$\Pi_{p,h}$,~\eqref{eq:H1stab},
	we obtain
    \begin{align*}
			 \|\partial_x(u-\widehat\Pi_k u) \|_{L_2(\widehat{\Omega})} 
					 &\le  \|\partial_x(u-\Pi_{p,h}^x u) \|_{L_2(\widehat{\Omega})} +
			 				\|\partial_x\Pi_{p,h}^x( u-\Pi_{p,h}^y u) \|_{L_2(\widehat{\Omega})}\\
					&  \le  \|\partial_x(u-\Pi_{p,h}^x u) \|_{L_2(\widehat{\Omega})} +
			 				\|\partial_x( u-\Pi_{p,h}^y u) \|_{L_2(\widehat{\Omega})} .
    \end{align*}	
    Using      Theorems~\ref{thrm:approx1} and~\ref{thrm:approx1a}, we obtain further
    \begin{align*}
			 \|\partial_x(u-\widehat\Pi_{k} u) \|_{L_2(\widehat{\Omega})} 			 				
					 &\le \sqrt{2} h  \|\partial_{xx} u \|_{L_2(\widehat{\Omega})} +
			 				\sqrt{2} h \|\partial_{xy} u \|_{L_2(\widehat{\Omega})} \\
					 &\le 2 h (  \|\partial_{xx} u \|_{L_2(\widehat{\Omega})}^2 +
			 				\|\partial_{xy} u \|_{L_2(\widehat{\Omega})}^2  )^{1/2}.
    \end{align*}	
	Using $\widehat\Pi_k = \Pi_{p,h}^y\Pi_{p,h}^x$, we obtain using the same arguments also
    \begin{align*}
			 \|\partial_y(u-\widehat \Pi_{k} u) \|_{L_2(\widehat{\Omega})} 
					& \le 2  h (  \|\partial_{xy} u \|_{L_2(\widehat{\Omega})}^2 +
			 				\|\partial_{yy} u \|_{L_2(\widehat{\Omega})}^2  )^{1/2},
    \end{align*}
    which yields
    \begin{align*}
			& | u- \widehat \Pi_{k} u |_{H^1(\widehat{\Omega})}^2 
			  =
			\|\partial_x(u- \widehat \Pi_{k} u) \|_{L_2(\widehat{\Omega})}^2 
			 + \|\partial_y(u-\widehat \Pi_{k} u) \|_{L_2(\widehat{\Omega})}^2 \\
					&\qquad \le 4 h^2 ( \|\partial_{xy} u \|_{L_2(\widehat{\Omega})}^2 +
						2  \|\partial_{xy} u \|_{L_2(\widehat{\Omega})}^2 +
			 				 \|\partial_{yy} u \|_{L_2(\widehat{\Omega})}^2  )
			 				 = 4 h^2 |u|_{H^2(\widehat{\Omega})}^2
    \end{align*}
    and finishes the proof.\qed
\end{proof}

\begin{theorem}\label{thrm:bdy}
	For all $u\in H^2(\widehat{\Omega})$, we obtain that
	\begin{itemize}
		\item $u$ and $\widehat \Pi_k u$ coincide at the corners of $\widehat{\Omega}$ and
		\item $\widehat \Pi_k u$, restricted on any edge $\widehat{\Gamma}$ of $\widehat{\Omega}$,
		coincides with the projector $\Pi_{p,h}$, applied to the restriction of $u$ to
		that edge. So, e.g., for $\widehat{\Gamma}=\{0\} \times (0,1)$,
		\[
				( \widehat \Pi_k u)(0,\cdot) = 
				\Pi_{p,h} (u(0,\cdot))
		\]
		holds.
	\end{itemize}
\end{theorem}
\begin{proof}
		This is a direct consequence of Lemma~\ref{lem:bc} and~\eqref{eq:hat:pi:k:def}.\qed
\end{proof}

\subsection{The multi-patch case}\label{sec:projector:3}

Assume to have a fully matching multi-patch discretization as introduced in Section~\ref{sec:prelim}
and let
\[
	\mathcal{H}^2(\Omega) := \{ u\in H^1(\Omega) \;:\; u|_{\Omega_k}\in H^2(\Omega_k) \},
	\qquad
	\|u\|_{\mathcal{H}^2(\Omega)}^2 := \sum_{k=1}^K \|u\|_{H^2(\Omega_k)}^2
\]
be a usual bent Sobolev space with corresponding norm. We obtain that the projectors
$\widehat{\Pi}_k$ are compatible.
\begin{lemma}
	For each $u \in \mathcal{H}^2(\Omega)\cap H^1_0(\Omega)$, there is exactly one $u_h \in V_h$ such that
	\begin{equation}\label{eq:pihatdef0}
		u_h \circ G_k = \widehat{\Pi}_k(u \circ G_k) \quad \mbox{for all $k=1,\ldots,K$}.
	\end{equation}
\end{lemma}
\begin{proof}
	First observe that~\eqref{eq:pihatdef0} specifies the value of $u_h$ for all patches $\Omega_k$
	and that the definition coincides with the pull-back definition~\eqref{eq:vkdef} of $V_k$. So, we obtain uniqueness
	and we obtain that the restriction of $u_h$ to any patch $\Omega_k$ yields a function
	in $V_k$.
	It remains to show that $u_h\in H^1_0(\Omega)$, i.e., that it is continuous and that it satisfies the Dirichlet boundary conditions.
	Theorem~\ref{thrm:bdy} implies that  the projector $\widehat{\Pi}_k$ is interpolatory on
	vertices, so $u_h$ is continuous at the vertices. For edges, Theorem~\ref{thrm:bdy} implies that  the
	projector $\widehat{\Pi}_k$ coincides with the univariate interpolation, so $u_h$ is also continuous
	across the edges. This shows continuity.
	Finally, observe that $u$ satisfies by assumption the homogenous Dirichlet boundary conditions. Again, on the boundary
	$\widehat{\Pi}_k u$ coincides with the univariate interpolation. As $u\equiv 0$ can be represented
	exactly by means of splines, we obtain that the univariate interpolation and, therefore, also $u_h$ vanish on the
	boundary (satisfies the Dirichlet boundary conditions).
\qed\end{proof}

So, we define the operator $\widetilde{\Pi}_h:\mathcal{H}^2(\Omega)\cap H^1_0(\Omega)\rightarrow V_h$ such that
\begin{equation}\label{eq:pihatdef}
	(\widetilde{\Pi}_h u) \circ G_k = \widehat \Pi_k (u \circ G_k) 
	\quad \mbox{for all $k=1,\ldots,K$}.
\end{equation}

This projector $\widetilde{\Pi}_h$ satisfies a standard error estimate.
\begin{theorem}\label{thrm:approx4}
    For all $u\in \mathcal{H}^2(\Omega)$, grid sizes $h$ 
    and spline degrees $p\in\mathbb{N}$, we obtain
	\begin{equation*}
		|u-\widetilde{\Pi}_h u |_{H^1(\Omega)} \lesssim h  \|u\|_{\mathcal{H}^2(\Omega)}.
	\end{equation*}
\end{theorem}
\begin{proof}	
	  Assumption~\ref{ass:geoequiv} yields
	    	$\|w\|_{H^1(\Omega_k)} \lesssim \| w\circ G_k \|_{H^1(\widehat{\Omega})}$
	  and
	    	$\|w\circ G_k\|_{H^2(\widehat{\Omega})}\lesssim \| w \|_{H^2(\Omega_k)}$.
	  Using~\eqref{eq:pihatdef} and Theorem~\ref{thrm:approx2}, we obtain
		\begin{align*}
		 \|u-\widetilde{\Pi}_h u \|_{H^1(\Omega_k)} 
			&\lesssim \|(u-\widetilde{\Pi}_h u)\circ G_k \|_{H^1(\widehat{\Omega})} 
			\le  \|u\circ G_k-\widehat\Pi_{k} (u\circ G_k) \|_{H^1(\widehat{\Omega})} \\
			&
			\lesssim h \|w\circ G_k\|_{H^2(\widehat{\Omega})} 
			\lesssim h \|u\|_{H^2(\Omega_k)}.
		\end{align*}
		By taking the sum over all patches, we obtain the desired result.\qed
\end{proof}

Obviously, the projector $\widetilde{\Pi}_h$ is not the $H^1$-orthogonal projector, but
the estimate for the $H^1$-orthogonal projection immediately follows. Note that $|\cdot|_{H^1(\Omega)}$
is a norm on $V_h$, so the following definition guarantees uniqueness.
\begin{definition}
	The projector $\Pi_{h}:\,H^1(\Omega)\rightarrow V_ {h}$ is the $H^1$-orthogonal
	projection, i.e., for any $u\in H^1(\Omega)$, the spline $u_{h}:=\Pi_{h} u$ satisfies
	\begin{equation}\nonumber
		(u-u_{h},v_{h})_{H^1(\Omega)} = 0 \quad \mbox{ for all } v_{h} \in V_h.
	\end{equation}
\end{definition}
\begin{theorem}\label{thrm:approx5}
		For all $u\in \mathcal{H}^2(\Omega)$, grid sizes $h$ 
    		and spline degrees $p\in\mathbb{N}$, we obtain
		\begin{equation*}
			|u-\Pi_h u |_{H^1(\Omega)} \lesssim h |u|_{\mathcal{H}^2(\Omega)}.
		\end{equation*}
\end{theorem}
\begin{proof}
	The minimization property of the projector  and 
	 Theorem~\ref{thrm:approx4} yields
	\[
		|u-\Pi_h u |_{H^1(\Omega)} \lesssim h \|u\|_{\mathcal{H}^2(\Omega)}.
	\]
		The Poincare inequality yields further
		\begin{equation*}
			|v-{\Pi}_h v |_{H^1(\Omega)}\lesssim h ( |v|_{\mathcal{H}^2(\Omega)} +(v,1)_{L_2(\Omega)})
		\end{equation*}
		for all $v \in \mathcal{H}^2(\Omega)$, so also for $v:=u-(u,1)_{L_2(\Omega)}$. 
		As $(I-{\Pi}_h) ( u-(u,1)_{L_2(\Omega)} )
		= (I-{\Pi}_h)u$ and $|u-(u,1)_{L_2(\Omega)}|_{\mathcal{H}^2(\Omega)}=|u|_{\mathcal{H}^2(\Omega)}$, this
	  finishes the proof.
	\qed
\end{proof}

Using a standard full elliptic regularity result,  we obtain also a corresponding $L_2-H^1$-estimate.
\begin{assumption}\label{ass:reg}
	For every $f\in L_2(\Omega)$,
	the solution $u\in H^1_0(\Omega)$ of the model problem~\eqref{eq:model} satisfies
	\begin{equation}\nonumber
		u\in H^{2}(\Omega) \qquad \mbox{and} \qquad |u|_{H^{2}(\Omega)} \le C_R \|f\|_{L_2(\Omega)}.
	\end{equation}
\end{assumption}
Such an estimate is satisfied for domains with smooth boundary, cf.~\cite{Necas:1967}, and for convex
polygonal domains, cf.~\cite{Dauge:1988,Dauge:1992}. In all cases, the constant $C_R$ only depends on
the shape of the computational domain $\Omega$, so $C_R \lesssim 1$.
\begin{theorem}\label{thrm:approx6}
		Assume to have Assumption~\ref{ass:reg}. Then, for all $u\in \mathcal{H}^2(\Omega)$, grid sizes $h$ 
    		and spline degrees $p\in\mathbb{N}$, we obtain
		\begin{equation}
			\|u-\Pi_h u\|_{L_2(\Omega)} \lesssim h |u|_{H^1(\Omega)}.
		\end{equation}
\end{theorem}
\begin{proof}
	This estimate is shown by a classical Aubin Nitsche duality trick.
	Let $v\in H^1_0(\Omega)$ be such that 
	\[
	   (v,w)_{H^1(\Omega)} = (u-\Pi_h u,w)_{L_2(\Omega)}\quad \mbox{for all }
	   			w \in H^1(\Omega).
	\]
	Observe that Assumption~\ref{ass:reg} implies
	$v\in H^2(\Omega)$ and $|v|_{H^2(\Omega)} = |v|_{\mathcal{H}^2(\Omega)} \lesssim
			 \|u-\Pi_h u\|_{L_2(\Omega)}$.
	Using this and Theorem~\ref{thrm:approx5}, we obtain 
	\begin{align*}
		& \|u-\Pi_h u \|_{L_2(\Omega)}
			 =  \frac{ (u-\Pi_h u,u-\Pi_h u)_{L_2(\Omega)} }{ \|u-\Pi_h u\|_{L_2(\Omega)} } \lesssim  \frac{ (u-\Pi_h u,v)_{H^1(\Omega)} }{|v|_{H^2(\Omega)} } 
			 \\
			& \quad 
			 \le  \sup_{w \in H^2(\Omega)}  \frac{ (u-\Pi_h u,w)_{H^1(\Omega)} }{ |w|_{H^2(\Omega)} }
			 \lesssim h \sup_{w \in H^2(\Omega)}  \frac{ (u-\Pi_h u,w)_{H^1(\Omega)} }{ |w-\Pi_h w |_{H^1(\Omega)} } .
	\end{align*}
	The $H^1$-orthogonality of the projector
	and the Cauchy-Schwarz inequality imply 
	\begin{align*}
		\|u-\Pi_h u \|_{L_2(\Omega)}
			& \lesssim\; h \sup_{w \in H^2(\Omega)}  \frac{ (u-\Pi_h u,w-\Pi_h w)_{H^1(\Omega)} }
			{ |w-\Pi_h w |_{H^1(\Omega)} }\\
			&\le  h |u-\Pi_h u|_{H^1(\Omega)} 
			\le h |u|_{H^1(\Omega)} ,
	\end{align*}
	which was to show.\qed
\end{proof}

\section{A multigrid solver}\label{sec:mg}

In this section, we develop a robust multigrid method for solving the linear system~\eqref{eq:linear:system}.
We assume to have a hierarchy of grids obtained by uniform refinement. For two consecutive
grid levels ($H=2h$), we have $V_H \subset V_h$, i.e., nested discretizations. For those,
we define $I_{H}^h$ to be the canonical embedding from $V_H$ into $V_h$ and the
restriction matrix $I_h^{H}$ to be its transpose.

Starting from an initial approximation~$\ul{u}_h^{(0)}$,
the next iterate $\ul{u}_h^{(1)}$ is obtained by the following two steps:
\begin{itemize}
    \item \emph{Smoothing:}
        For some fixed number $\nu$ of smoothing steps, compute
              \begin{equation} \label{eq:sm}
                   \ul{u}_h^{(0,\mu)} := \ul{u}_h^{(0,\mu-1)} + \tau L_h^{-1}
                                    \left(\ul{f}_h-A_h\;\ul{u}_h^{(0,\mu-1)}\right)
                                    \qquad \mbox{for } \mu=1,\ldots,\nu,
              \end{equation}
              where $\ul{u}^{(0,0)} := \ul{u}^{(0)}$. The choice of
              the matrix $L_h$ and of the damping parameter $\tau>0$ will be discussed below.
    \item \emph{Coarse-grid correction:}
        \begin{itemize}
             \item Compute the defect and restrict it to the coarser grid:
                \[
                      \ul{r}_{H}^{(1)} := I_h^{H} \left(\ul{f}_h - A_h
                      \;\ul{u}_h^{(0,\nu)}\right).
                \]
             \item Compute the correction $\ul{p}_{H}^{(1)}$ by approximately solving the coarse-grid problem
                \begin{equation}\label{eq:coarse:grid:problem}
                    A_{H} \,\ul{p}_{H}^{(1)} =\ul{r}_{H}^{(1)}.
                \end{equation}
             \item Prolongate $\ul{p}_{H}^{(1)}$  and add the result to the previous iterate:
                  \[
                       \ul{u}_h^{(1)} := \ul{u}^{(0,\nu)} +
                        I_{H}^h \, \ul{p}_{H}^{(1)}.
                  \]
        \end{itemize}
\end{itemize}
If the problem \eqref{eq:coarse:grid:problem} on the coarser grid is solved exactly
(\emph{two-grid method}), the coarse-grid correction is given by
\begin{equation} \label{eq:method:cga}
        \ul{u}_h^{(1)} := \ul{u}_h^{(0,\nu)} +
        I_{H}^{h} \, A_{H}^{-1} \,  I_{h}^{H}
        \left( \ul{f}_h - A_h \;\ul{u}_h^{(0,\nu)}\right).
\end{equation}
In practice, the problem~\eqref{eq:coarse:grid:problem} is
approximately solved by recursively applying one step (\emph{V-cycle})
or two steps (\emph{W-cycle}) of the multigrid method. On
the coarsest grid level, the problem~\eqref{eq:coarse:grid:problem} is
solved exactly using a direct method.

\subsection{An additive smoother}\label{sec:additive}

For the single-patch case, we have proposed the \emph{subspace-corrected mass smoother}
in~\cite{HT:2016}.
For the multi-patch case, we propose
\begin{equation}\label{eq:additive}
	L_h := \sum_{T\in \mathbb{T}} P_T L_T P_T^{\top},
\end{equation}
where $P_T$ and $L_T$ are chosen as follows.
\begin{itemize}
	\item The matrices $P_T$ represent the canonical embedding from $\Phi^{(T)}$ in $\Phi$.
	By construction, this is a full-rank $N\times |\Phi^{(T)}|$  binary matrix, where each column has
	exactly one non-zero entry.
	\item $L_T$ are local smoothers. For $T\in \mathbb K$,
	we choose $L_T^{-1}$ to be the \emph{subspace-corrected mass smoother}.
	For $T\in \mathbb E \cup \mathbb V$, we choose
	\begin{equation}\label{eq:small:def}
		L_T:= P_T^{\top} A_h P_T,
	\end{equation}
	i.e., $L_T^{-1}$ is an exact solver.
\end{itemize}
This choice of $L_T$ is feasible because for any $T\in \mathbb E$, the matrix $L_T$ has a dimension
of $\mathcal{O}(n)$ and for any $T \in \mathbb E$ the matrix $L_T$ is just a 1-by-1 matrix.
Note that the construction of the subspace corrected mass smoother requires for each patch
that $m>p$, i.e., that the number of intervals per direction is larger than $p$; for patches
where this is not satisfied, one can choose $L_T:=P_T^{\top} A_h P_T$.

Note that the matrices $P_T$ realize a partition of the degrees of freedom (like a patch-wise Jacobi iteration), so 
$L_h$ is a (in general: reordered) block-diagonal matrix that can be inverted by inverting the blocks. So, we obtain
\[
         \ul{u}_h^{(0,\mu)} := \ul{u}_h^{(0,\mu-1)} + \tau 
         	\underbrace{\sum_{T\in \mathbb T} P_T L_T^{-1} P_T^{\top} }_{\displaystyle L_h^{-1}=}
                      \left(\ul{f}_h-A_h\;\ul{u}_h^{(0,\mu-1)}\right).
\]
In~\cite{HT:2016}, we have shown for the the single-patch case that a multigrid solver with the
subspace-corrected mass smoother converges robustly. Here, we recall these results,
where the presentation of the results is slightly altered such that we can prove the results
for the multi-patch case smoothly in the sequel.

The following theorem is a slight variation of the standard multigrid theory as developed by Hackbusch~\cite{Hackbusch:1985}.
\begin{theorem}\label{thrm:conv}
	Assume that the conditions of Theorem~\ref{thrm:approx6} hold and that $L_h$ satisfies
	\begin{equation}\label{eq:thrm:conv}
			\ul{c} A_h \le L_h \le \ol{c} (A_h + h^{-2} M_h).
	\end{equation}
	Then the two-grid method converges for the choice $\tau\in(0, \ul{c}\,]$ and $\nu > \nu_0 := \tau^{-1} \ol{c}(1+4c_A^2)$
	with rate $q = \nu_0/\nu$, i.e.,
	\[
		\| \ul{u}_h^{(1)} - A_h^{-1} \ul{f}_h \|_{A_h + h^{-2} M_h} \le q \| \ul{u}_h^{(0)} - A_h^{-1} \ul{f}_h \|_{A_h + h^{-2} M_h},
	\]
	where $c_A$ is the constant hidden in the estimate in Theorem~\ref{thrm:approx6}.
\end{theorem}
\begin{proof}
	We use~\cite[Theorem~3]{HTZ:2016}. 
	First, observe that Theorem~\ref{thrm:approx6} implies 
	\[
		\|(I-\Pi_H)\ul{u}_h\|_{M_h}^2 \le c_A^2 H^2 \|\ul{u}_h \|_{A_h}^2 \le 4 c_A^2 h^2 \|\ul{u}_h \|_{A_h}^2,
	\]
	where $\Pi_H$ is the $A_h$-orthogonal projector or, equivalently, the $H^1$-orthogonal projector.
	Because projectors are stable, we also obtain
	\[
		\|(I-\Pi_H)\ul{u}_h\|_{A_h}^2 \le \|\ul{u}_h\|_{A_h}^2,
	\]
	and using~\eqref{eq:thrm:conv} also
	\[
		\|(I-\Pi_H)\ul{u}_h\|_{L_h}^2 \le \ol{c} \|(I-\Pi_H)\ul{u}_h\|_{A_h + h^{-2} M_h}^2 \le \ol{c} (1+4 c_A^2) \|\ul{u}_h\|_{A_h}^2,
	\]
	i.e., the first condition (approximation error estimate)
	in \cite[Theorem~3]{HTZ:2016} with $C_A=\ol{c} (1+4 c_A^2)$. Now, observe that the first inequality in~\eqref{eq:thrm:conv} 
	coincides with second condition (inverse inequality) in \cite[Theorem~3]{HTZ:2016} with $C_I =  \ul{c}^{-1}$.
	Finally, \cite[Threorem~3]{HTZ:2016} shows the desired statement.
\qed\end{proof}

In~\cite[Theorem~4]{HTZ:2016}, it was shown that under the assumptions of \cite[Theorem~3]{HTZ:2016} also
a W-cycle multigrid method converges.

Now, we show that the conditions of Theorem~\ref{thrm:conv} hold patch-wise for the
subspace-corrected mass smoother. For this purpose, we define the piece-local stiffness and mass matrices by
\[
	A_T := P_T^{\top} A_h P_T \qquad\mbox{and} \qquad M_T := P_T^{\top} M_h P_T.
\]
Remember that the domain $\Omega$ consists of the patches $\Omega_k$ for $k=1,\ldots,K$. So,
we define $A_k$ and $M_k$ to be the stiffness and mass matrix obtained by restricting the integration
to the patches, i.e.,
\[
		A_k := [ (\nabla \phi_i, \nabla \phi_j)_{L_2(\Omega_k)} ]_{i,j=1}^N
		\qquad  \mbox{and} \qquad
		M_k := [ (\phi_i, \phi_j)_{L_2(\Omega_k)} ]_{i,j=1}^N
\]
and observe
\begin{equation}\label{eq:patchwise:assembling}
		A_h = \sum_{k=1}^K A_k	\qquad  \mbox{and} \qquad
		M_h = \sum_{k=1}^K M_k.
\end{equation}
Analogously to $A_T$ and $M_T$, we define
$
	A_{k,T} := P_T^{\top} A_k P_T   \mbox{ and } 
	M_{k,T} := P_T^{\top} M_k P_T.
$
Finally, we define stiffness and mass matrices on the parameter domain by
\[
		\widehat{A}_k := [ (\nabla (\phi_i\circ G_k), \nabla (\phi_j\circ G_k))_{L_2(\widehat{\Omega})} ]_{i,j=1}^N, \qquad
		\widehat{M}_k := [ (\phi_i\circ G_k, \phi_j\circ G_k)_{L_2(\widehat{\Omega})} ]_{i,j=1}^N,
\]
\[
		\widehat{A}_h := \sum_{k=1}^K \widehat{A}_k, \quad
		\widehat{M}_h: = \sum_{k=1}^K \widehat{M}_k, \quad
	\widehat{A}_{k,T} := P_T^{\top} \widehat{A}_k P_T,   \mbox{ and} \quad 
	\widehat{M}_{k,T} := P_T^{\top} \widehat{M}_k P_T
\]
and observe that they are similar to the corresponding matrices on the physical domain.

\begin{lemma}\label{lem:geoequiv}
	We have
	\[ A_k \eqsim \widehat{A}_k, \qquad 
		A_h \eqsim \widehat{A}_h, \qquad 
		A_{T} \eqsim \widehat{A}_{T}, \qquad 
		A_{k,T} \eqsim \widehat{A}_{k,T}, \]
	and analogous results for $M_k$, $M_h$, $M_{T}$, and $M_{k,T}$.
\end{lemma}
\begin{proof}
	We have using Assumption~\ref{ass:geoequiv}
	\[
		\|\ul{u}_h \|_{A_k}^2 =  \|u_h \|_{H^1(\Omega_k)}^2 \eqsim
			\|u_h\circ G_k \|_{H^1(\widehat{\Omega})}^2 = \|\ul{u}_h \|_{\widehat{A}_k}^2,
	\]
	which shows the first statement. The second one is obtained by summing over $k$, the third
	one is obtained as $A_h \eqsim \widehat{A}_h$ implies
	$A_{T} = P_T^{\top} A_h P_T \eqsim P_T^{\top} \widehat{A}_h P_T = \widehat{A}_{T}$, and
	the fourth is obtained as $A_k \eqsim \widehat{A}_k$ implies
	$A_{k,T} = P_T^{\top} A_k P_T \eqsim P_T^{\top} \widehat{A}_k P_T = \widehat{A}_{k,T}$.
	The statements for the mass matrix are completely analogous.
\qed\end{proof}

The following Lemma follows directly from what has been shown in~\cite[Section~4.2]{HT:2016}.
\begin{lemma}\label{lem:equiv}
	For all grid sizes $h$ and spline degrees $p\in\mathbb{N}$, the relation
	\begin{equation}\label{eq:lem:equiv}
		A_T \lesssim L_T \lesssim  A_T + h^{-2} M_T
		\qquad \mbox{	holds for all $T\in \mathbb T$.}
	\end{equation}
\end{lemma}
\begin{proof}
	For $T\in \mathbb T$, the estimate has been shown in the proofs of \cite[Lemmas~8 and 9]{HT:2016}.
	For $T \in \mathbb E\cup \mathbb V$, we have $L_T=A_T$, so the desired statement immediately follows.
\qed\end{proof}

Now we show that $L_h$, as defined in~\eqref{eq:additive}, satisfies the condition of
Theorem~\ref{thrm:conv} with $\ul{c}$ being robust
and with $\ol{c}$ depending linearly on the spline degree, i.e.,
\begin{equation}\label{eq:what:to:show}
	A_h \lesssim L_h \lesssim p (A_h + h^{-2} M_h).
\end{equation}
We show this by showing
\begin{align}
	\label{eq:what:to:show:1}
	 A_h \lesssim  &\sum_{T\in \mathbb{T}} P_T A_T P_T^{\top}, &\\
	\label{eq:what:to:show:0}
	&\sum_{T\in \mathbb{T}} P_T A_T P_T^{\top} \lesssim L_h  \lesssim  \sum_{T\in \mathbb{T}} P_T  (A_T + h^{-2} M_T) P_T^{\top}, \\
	\label{eq:what:to:show:2}
	&\hspace{3.4cm} \sum_{T\in \mathbb{T}} P_T  (A_T + h^{-2} M_T) P_T^{\top} \lesssim p  (A_h + h^{-2} M_h).\hspace{-.2cm}
\end{align}
Note that~\eqref{eq:what:to:show:0} follows directly from~\eqref{eq:additive} and Lemma~\ref{lem:equiv}.
The other two inequalities are shown in the sequel.
\begin{lemma}\label{lem:first}
		For all grid sizes $h$ and spline degrees $p\in\mathbb{N}$, 
		the inequality~\eqref{eq:what:to:show:1} 	holds.
\end{lemma}
\begin{proof}
	Using~$\sum_{T\in \mathbb T} P_T  P_T^{\top} = I$, we obtain
	\[
		\|\ul u_h\|_{A_h}^2
			= 
			\Big\|\sum_{T\in \mathbb{T}} P_T P_T^{\top} \ul u_h \Big\|_{A_h }^2
			= 
			\sum_{T\in \mathbb{T}}\sum_{S\in \mathbb{T}} ( P_T^{\top}A_h P_S P_S^{\top} \ul u_h, P_T^{\top} \ul u_h).
	\]
	Note that Assumption~\ref{ass:neigbour} implies that for any $T\in \mathbb{T}$, the number of
	$S\in \mathbb{T}$ such that $P_T^{\top}A_hP_S\not=0$ is bounded. So, we obtain using the Cauchy-Schwarz
	inequality that
	\[
		\|\ul u_h\|_{A_h}^2
			\lesssim
			\sum_{T\in \mathbb{T}} \|P_T^{\top} \ul u_h\|_{P_T^{\top}A_hP_T}^2 = \sum_{T\in \mathbb{T}} \|P_T^{\top} \ul u_h\|_{A_T}^2,
	\]
	which finishes the proof.
\qed\end{proof}

For showing \eqref{eq:what:to:show:2}, we need some trace estimates. 
The following lemma is a standard result, which is given to keep the paper self-contained.

\begin{lemma}\label{lem:a}
	$|u(0)|^2 \le \|u\|_{L_2(0,1)}^2 + \|u\|_{L_2(0,1)} |u|_{H^1(0,1)}$ holds for all $u\in H^1(0,1)$.
\end{lemma}
\begin{proof}
	Let $u\in H^1(0,1)$ be arbitrary but fixed and note that $u$ is continuous.
	We have for all $t\in (0,1)$ that
	\begin{align*}
			|u(0)|^2 &= -\int_0^t u(s) u'(s) \mbox{d} s + |u(t)|^2 
	\end{align*}
	holds. So,		
	\begin{align*}
			&|u(0)|^2  \\&= - \int_0^1 \int_0^t u(s) u'(s) d s + |u(t)|^2 \mbox{d} t 
							 \le  \int_0^1 \|u\|_{L_2(0,s)} \|u'\|_{L_2(0,s)} \mbox{d}t + \|u\|_{L_2(0,1)}^2 \\
							& \le  \int_0^1 \|u\|_{L_2(0,1)} \|u'\|_{L_2(0,1)} \mbox{d}t + \|u\|_{L_2(0,1)}^2 
							 = \|u\|_{L_2(0,1)} \|u'\|_{L_2(0,1)}  + \|u\|_{L_2(0,1)}^2,
	\end{align*}
	which finishes the proof.
\qed\end{proof}

Observe that on each patch $\Omega_k$, we obtain the following stability estimates.
\begin{lemma}\label{lem:stab:v}
	For all $k\in\{1,\ldots,K\}$ and all $T \in \mathbb V_k$, the inequality
	\[
		P_T (A_{k,T}+h^{-2} M_{k,T}) P_T^{\top} \lesssim p (A_k + h^{-2} M_k)
	\]
	holds.
\end{lemma}
\begin{proof}
	Let $k$ and $T$ be arbitrary but fixed. Note that the parameter domain was defined to be $\widehat{\Omega}=(0,1)^2$.
	Assume without loss of generality that that vertex $T$ corresponds to the vertex $\widehat{T} = (0,0)$
	on the parameter domain. Define $\widehat{\Gamma}:=\{0\} \times (0,1)$ to be an edge that touches that vertex.
	Define on the parameter domain the norms
	\begin{equation}\label{eq:lem:stab:v:norms}
		\begin{aligned}
		\|\widehat{u}_h\|_{Q(\widehat{\Omega})}^2 &:= |\widehat{u}_h |_{H^1(\widehat{\Omega})}^2 + h^{-2} \| \widehat{u}_h \|_{L_2(\widehat{\Omega})}^2, \\
		\|\widehat{u}_h\|_{Q(\widehat{\Gamma})}^2 &:= p^{-1} h |\widehat{u}_h |_{H^1(\widehat{\Gamma})}^2+ p h^{-1} \|\widehat{u}_h\|_{L_2(\widehat{\Gamma})}^2,
		\end{aligned} 
	\end{equation}
	and observe that Lemma~\ref{lem:geoequiv} implies
	\begin{equation}\label{eq:lem:stab:v:0}
		\| \ul{u}_h\|_{A_k + h^{-2} M_k}^2 \eqsim 
		\| \ul{u}_h\|_{\widehat{A}_k + h^{-2} \widehat{M}_k}^2 = \|\widehat{u}_h \|_{Q(\widehat{\Omega})}^2,
	\end{equation}
	where here and in what follows $\widehat{u}_h:=u_h\circ G_k$.

	Now we compute $\| P_T^{\top} \ul{u}_h\|_{A_{k,T}}$ and $\| P_T^{\top} \ul{u}_h\|_{M_{k,T}}$.
	Note that there is just one basis function assigned to the vertex. Due to the tensor-product structure, this basis function
	is 
	\[
		\widehat{B}_k^{(1)}(x,y) = \widehat{B}_{p,h}^{(1)}(x) \widehat{B}_{p,h}^{(1)}(y) = \max\{0,(1-x/h)^p\}\max\{0,(1-y/h)^p\}.
	\]
	As $\widehat{B}_k^{(1)}(0,0)=1$ and all other basis functions vanish on $(0,0)$, we obtain	
	\begin{equation}\label{eq:lem:stab:v:norms:2}
		\begin{aligned}
		\|P_T^{\top}\ul{u}_h\|_{A_{k,T}}^2 &\eqsim \|P_T^{\top}\ul{u}_h\|_{\widehat{A}_{k,T}}^2 = 2 |\widehat{B}_{p,h}^{(1)}|_{H^1(0,1)}^2 \|\widehat{B}_{p,h}^{(1)}\|_{L_2(0,1)}^2 |\widehat{u}_h(0,0)|^2,\\
		\|P_T^{\top}\ul{u}_h\|_{M_{k,T}}^2 &\eqsim \|P_T^{\top}\ul{u}_h\|_{\widehat{M}_{k,T}}^2 = \|\widehat{B}_{p,h}^{(1)}\|_{L_2(0,1)}^4 |\widehat{u}_h(0,0)|^2.
		\end{aligned} 
	\end{equation}
	Straight-forward computations yield
	\begin{equation}\label{eq:lem:stab:v:norms:2a}
		\|\widehat{B}_{p,h}^{(1)}\|_{L_2(0,1)}^2 = \frac{h}{2p+1} \eqsim \frac hp 
		\quad \mbox{and}\quad
		|\widehat{B}_{p,h}^{(1)}|_{H^1(0,1)}^2=   \frac{p^2}{h(2p-1)} \eqsim\frac ph.
	\end{equation}
	So,
	\begin{equation}\label{eq:lem:stab:v:1}
		\|P_T^{\top} \ul{u}_h \|_{A_{k,T}+ h^{-2} M_{k,T}}^2 
		\eqsim \left( 2 \frac hp \frac ph+\frac{h^2}{p^2} \right) |\widehat{u}_h(0,0)|^2
		\eqsim |\widehat{u}_h(0,0)|^2.
	\end{equation}
	 Observe that Lemma~\ref{lem:a}, and $ab\le a^2+b^2$ imply
	\begin{align}\nonumber
		|\widehat{u}_h(0,0)|^2 
			&\le \|\widehat{u}_h\|_{L_2(\widehat{\Gamma})}^2 + \|\widehat{u}_h\|_{L_2(\widehat{\Gamma})} |\widehat{u}_h|_{H^1(\widehat{\Gamma})} \\
			&\lesssim (1+p h^{-1}) \|\widehat{u}_h\|_{L_2(\widehat{\Gamma})}^2  + p^{-1} h |\widehat{u}_h |_{H^1(\widehat{\Gamma})}^2
			 \eqsim \|\widehat{u}_h\|_{Q(\widehat{\Gamma})}^2. \label{eq:lem:stab:v:2}
	\end{align}
	Now, we show
	\begin{equation}\label{eq:lem:stab:v:3}
		\|\widehat{u}_h\|_{Q(\widehat{\Gamma})}^2 \lesssim p \|\widehat{u}_h\|_{Q(\widehat{\Omega})}^2.
	\end{equation}
	Using Lemma~\ref{lem:a}, we immediately obtain
	\begin{align*}
		|\widehat{u}_h(0,y)|^2 &\le  \|\widehat{u}_h(\cdot,y)\|_{L_2(0,1)}^2
				+  \|\widehat{u}_h(\cdot,y)\|_{L_2(0,1)} |\widehat{u}_h(\cdot,y) |_{H^1(0,1)}.
	\end{align*}
	By integrating over $y$, using the Cauchy Schwarz inequality and $ab\le a^2+b^2$, we obtain further
	\begin{align}\nonumber
		\|\widehat{u}_h\|_{L_2(\Gamma)} ^2 
		&\le  \int_0^1 \|\widehat{u}_h(\cdot,y)\|_{L_2(0,1)}^2 + \|\widehat{u}_h(\cdot,y)\|_{L_2(0,1)} |\widehat{u}_h(\cdot,y)|_{H^1(0,1)} \mbox{d}y \\
		&\le \|\widehat{u}_h\|_{L_2(\Omega)} ^2 +  \|\widehat{u}_h\|_{L_2(\Omega)} \|\partial_x \widehat{u}_h\|_{L_2(\Omega)} \nonumber \\
		& \le (1+h^{-1}) \|\widehat{u}_h\|_{L_2(\Omega)} ^2 + h \|\partial_x \widehat{u}_h\|_{L_2(\Omega)}^2 \nonumber \\
		&\lesssim h | \widehat{u}_h|_{H^1(\Omega)}^2+ h^{-1} \|\widehat{u}_h\|_{L_2(\Omega)} ^2  
		. \label{eq:lem:v:1}
	\end{align}
	Analogously, we obtain
	\begin{align*}
		|\widehat{u}_h |_{H^1(\Gamma)} ^2 = \|\partial_y \widehat{u}_h\|_{L_2(\Gamma)} ^2
		\le \|\partial_y \widehat{u}_h\|_{L_2(\Omega)} ^2  + \|\partial_y \widehat{u}_h\|_{L_2(\Omega)} \|\partial_y \partial_x \widehat{u}_h\|_{L_2(\Omega)}. 
	\end{align*}
	Using a standard inverse inequality, cf.~\cite[Theorem~3.91]{Schwab:1998}, and $ab\le a^2+b^2$, we obtain further
	\begin{align}\label{eq:lem:v:2}
		|\widehat{u}_h |_{H^1(\Gamma)} ^2 
		&\lesssim \|\partial_y \widehat{u}_h\|_{L_2(\Omega)} ^2 + p^2 h^{-1}  \|\partial_y \widehat{u}_h\|_{L_2(\Omega)} \| \partial_x \widehat{u}_h\|_{L_2(\Omega)}
		\lesssim  p^2 h^{-1}  | \widehat{u}_h |_{H^1(\Omega)}^2. 
	\end{align}
	By combining~\eqref{eq:lem:stab:v:norms}, \eqref{eq:lem:v:1} and~\eqref{eq:lem:v:2}, we obtain
	\begin{align*}
		\|\widehat{u}_h  \|_{Q(\widehat{\Gamma})}^2 &\lesssim
			p | \widehat{u}_h  |_{H^1(\Omega)} ^2 +  ph^{-2} \|\widehat{u}_h\|_{L_2(\Omega)} ^2 
		= p \|\widehat{u}_h\|_{Q(\widehat{\Omega})}^2,
	\end{align*}
	which finishes the proof of~\eqref{eq:lem:stab:v:3}.
	Using~\eqref{eq:lem:stab:v:1}, \eqref{eq:lem:stab:v:2}, \eqref{eq:lem:stab:v:3} and~\eqref{eq:lem:stab:v:0}, we obtain
	\begin{align*}
		\| P_T^{\top} \ul{u}_h \|_{A_{k,T}+h^{-2} M_{k,T}}^2 & \eqsim
		|\widehat{u}_h(0,0)|^2 \lesssim \|\widehat{u}_h\|_{Q(\widehat{\Gamma})}^2  \\
			& \lesssim p \|\widehat{u}_h\|_{Q(\widehat{\Omega})}^2 \eqsim p \| \ul{u}_h\|_{A_k + h^{-2} M_k}^2,
	\end{align*}
	which finishes the proof.
\qed\end{proof}
\begin{lemma}\label{lem:stab:e}
	For all $k\in\{1,\ldots,K\}$ and all $T \in \mathbb E_k$, the inequality
	\[
		P_T (A_{k,T}+h^{-2} M_{k,T}) P_T^{\top} \lesssim p (A_k + h^{-2} M_k)
	\]
	holds.
\end{lemma}
\begin{proof}
	Let $k$ and $T$ be arbitrary but fixed. Note that the parameter domain was defined to be $\widehat{\Omega}=(0,1)^2$.
	Assume without loss of generality that that edge $T$ corresponds to the edge $\widehat{\Gamma}:=\{0\} \times (0,1)$
	on the parameter domain.
	We define on the parameter domain the norms
	$\|\widehat{u}_h\|_{Q(\widehat{\Omega})}^2 $ and $\|\widehat{u}_h\|_{Q(\widehat{\Gamma})}^2 $ as in~\eqref{eq:lem:stab:v:norms}
	and use again $\widehat{u}_h:= u_h \circ G_k$.

	Due to the tensor-product structure, the basis functions contributing to the edge have the form
	\[
		\widehat{B}_k^{(i)}(x,y) = \widehat{B}_{p,h}^{(1)}(x) \widehat{B}_{p,h}^{(i)}(y) = \max\{0,(1-x/h)^p\}\widehat{B}_{p,h}^{(i)}(y)
		\;\mbox{ for $i=1,\ldots,n$. }
	\]
	Note that among those, the first and the last one are associated to the corresponding vertices
	$(0,0)$ and $(0,1)$. Only the basis functions in between belong to $\Phi^{(T)}$. 
	Analogously to~\eqref{eq:lem:stab:v:norms:2}, we have
	\begin{align} \nonumber
		\|P_T^{\top}\ul{u}_h\|_{A_{k,T}}^2 &\eqsim \|P_T^{\top}\ul{u}_h\|_{\widehat{A}_{k,T}}^2\\
		&= | \widehat{u}_h - \widehat{u}_h(0,0) \widehat{B}_{p,h}^{(1)} - \widehat{u}_h(0,1) \widehat{B}_{p,h}^{(n)}  |_{H^1(\widehat{\Gamma})}^2 \|\widehat{B}_{p,h}^{(1)}\|_{L_2(0,1)}^2\nonumber \\
		&\quad + \| \widehat{u}_h - \widehat{u}_h(0,0) \widehat{B}_{p,h}^{(1)} - \widehat{u}_h(0,1) \widehat{B}_{p,h}^{(n)}  \|_{L_2(\widehat{\Gamma})}^2 |\widehat{B}_{p,h}^{(1)} |_{H^1(0,1)}^2,\label{eq:lem:stab:e:res:1}\\
		\|P_T^{\top}\ul{u}_h\|_{M_{k,T}}^2 & \eqsim \|P_T^{\top}\ul{u}_h\|_{\widehat{M}_{k,T}}^2 \nonumber\\
		&= \| \widehat{u}_h - \widehat{u}_h(0,0) \widehat{B}_{p,h}^{(1)} - \widehat{u}_h(0,1) \widehat{B}_{p,h}^{(n)}  \|_{L_2(\widehat{\Gamma})}^2 \|\widehat{B}_{p,h}^{(1)}\|_{L_2(0,1)}^2,\label{eq:lem:stab:e:res:2}
	\end{align}
	where superfluous contributions from the vertices have been subtracted.
	Again, using the triangle inequality and~\eqref{eq:lem:stab:v:norms:2a}, we obtain
	\begin{align*}
		\|P_T^{\top}\ul{u}_h\|_{A_{k,T}+ h^{-2} M_{k,T}}^2 
			 &= \frac{h}{p} | \widehat{u}_h - \widehat{u}_h(0,0) \widehat{B}_{p,h}^{(1)} - \widehat{u}_h(0,1) \widehat{B}_{p,h}^{(n)}  |_{H^1(\widehat{\Gamma})}^2\\ 
			&\quad +\frac{p}{h} \| \widehat{u}_h - \widehat{u}_h(0,0) \widehat{B}_{p,h}^{(1)} - \widehat{u}_h(0,1) \widehat{B}_{p,h}^{(n)}  \|_{L_2(\widehat{\Gamma})}^2\\
			&  \lesssim \frac{h}{p} | \widehat{u}_h |_{H^1(\widehat{\Gamma})}^2  +\frac{p}{h} \| \widehat{u}_h \|_{L_2(\widehat{\Gamma})}^2
				+ |\widehat{u}_h(0,0) |^2 + |\widehat{u}_h(0,1)|^2.
	\end{align*}
	Using the definition of $\|\widehat{u}_h \|_{Q(\widehat{\Gamma})}$ and~\eqref{eq:lem:stab:v:2}, we obtain further
	\[
		\|P_T^{\top}\ul{u}_h\|_{A_{k,T}+ h^{-2} M_{k,T}}^2 \lesssim \|\widehat{u}_h \|_{Q(\widehat{\Gamma})}^2
	\]
	and using
	\eqref{eq:lem:stab:v:3} and~\eqref{eq:lem:stab:v:0} finally
	\begin{align*}
		\| P_T^{\top} \ul{u}_h \|_{A_{k,T}+ h^{-2} M_{k,T}}^2 
		\lesssim p \|\widehat{u}_h\|_{Q(\widehat{\Omega})}^2 
		\eqsim p \| \ul{u}_h\|_{A_k + h^{-2} M_k}^2,
	\end{align*}
	which finishes the proof.
\qed\end{proof}

\begin{lemma}\label{lem:third}
		For all grid sizes $h$ and spline degrees $p\in\mathbb{N}$,
		the inequality~\eqref{eq:what:to:show:2} holds.
\end{lemma}
\begin{proof}
	Let $k$ be arbitrary but fixed. Observe that $\mathbb{T}_k = \mathbb{K}_k \cup
	\mathbb{E}_k \cup \mathbb{V}_k $ and that $\mathbb{K}_k = \{\Omega_k\}$. Certainly, the number
	of edges and the number of vertices do not exceed $4$ (they are smaller if the patch $\Omega_k$ contributes
	to the (Dirichlet) boundary), so $|\mathbb{E}_k \cup \mathbb{V}_k|\le 8$ holds.
	Analogously to the proof of Lemma~\ref{lem:first}, we obtain
	\[
		\|\ul{u}_h\|_{A_{k,\Omega_k}+ h^{-2} M_{k,\Omega_k}}^2
		\lesssim  \|\ul{u}_h\|_{A_k+h^{-2} M_k} ^2 + \sum_{T \in \mathbb{E}_k \cup \mathbb{V}_k } \|\ul{u}_h\|_{A_{k,T} + h^{-2} M_{k,T}}^2
	\]
	and, as $\mathbb T_k = \{\Omega_k\} \cup \mathbb{E}_k \cup \mathbb{V}_k$,
	\[
		\sum_{T\in \mathbb T_k} \|\ul{u}_h\|_{A_{k,T}+ h^{-2} M_{k,T}}^2 
			\lesssim  \|\ul{u}_h\|_{A_k+h^{-2} M_k} ^2 + \sum_{T \in \mathbb{E}_k \cup \mathbb{V}_k } \|\ul{u}_h\|_{A_{k,T} +  h^{-2} M_{k,T}}^2.
	\]
	Using Lemmas~\ref{lem:stab:v} and~\ref{lem:stab:e} and $|\mathbb{E}_k \cup \mathbb{V}_k|\le 8$, we obtain also
	\[
		\sum_{T\in \mathbb T_k} \|\ul{u}_h\|_{A_{k,T}+ h^{-2} M_{k,T}}^2 
				\lesssim  p \|\ul{u}_h\|_{A_{k}+h^{-2} M_{k}}^2.
	\]
	By adding this up over all patches, we obtain using~\eqref{eq:patchwise:assembling} that
	\begin{align*}
		\sum_{T\in \mathbb T} \|\ul{u}_h\|_{A_{T}+ h^{-2} M_{T}}^2
		& =
		\sum_{k=1}^K
		\sum_{T\in \mathbb T_k} \|\ul{u}_h\|_{A_{k,T}+ h^{-2} M_{k,T}}^2
		 \lesssim
		\sum_{k=1}^K
		p \|\ul{u}_h\|_{A_{k}+h^{-2} M_{k}}^2\\
		&=
		p \|\ul{u}_h\|_{A+h^{-2} M}^2,
	\end{align*}
	which finishes the proof.
\qed\end{proof}

\begin{lemma}\label{lem:last}
		For all grid sizes $h$, and spline degrees $p\in\mathbb{N}$,
		the inequality~\eqref{eq:what:to:show} holds.
\end{lemma}
\begin{proof}
	This is just the combination of the Lemmas~\ref{lem:first}, \ref{lem:equiv} and~\ref{lem:third}.
\qed\end{proof}

Based on this, we can show that the multigrid solver converges robustly 
if $\mathcal{O}(p)$ smoothing steps are applied.
\begin{theorem}
	There are constants $c_1$ and $c_2$ that do not depend on  the grid size $h$, the spline degree $p$,  and the number of
	patches $K$ (but may depend on $C_G$, $C_N$, or $C_R$) such that 
	\begin{equation}\label{eq:thrm:final1}
			\tau L_h^{-1} A_h \le 1
	\end{equation}
	for all $\tau \in (0,c_1]$ and the proposed	two-grid method converges for
	any $\tau$ satisfying~\eqref{eq:thrm:final1} and any choice of the number of smoothing steps
	$\nu > \nu_0 := p \tau^{-1}  c_2$
	with a  convergence rate $q = \nu_0/\nu$, i.e.,
	\[
		\| \ul{u}_h^{(1)} - A_h^{-1} \ul{f}_h \|_{A_h + h^{-2} M_h} \le \frac{\nu_0}{\nu}
			\| \ul{u}_h^{(0)} - A_h^{-1} \ul{f}_h \|_{A_h + h^{-2} M_h}.
	\]
\end{theorem}
\begin{proof}
	We use Theorem~\ref{thrm:conv}, whose condition is shown by Lemma~\ref{lem:last}.
\qed\end{proof}
Due to~\cite[Theorem~4]{HTZ:2016}, we know that also the W-cycle multigrid method converges.

\begin{remark}\label{rem:1}
	Because the computational
	costs for the (exact) solvers for the edges and the vertices are negligible, we obtain
	that the overall computational complexity coincides with that of the subspace corrected
	mass smoother, as computed in~\cite[Section~5.4]{HT:2016},
	multiplied with the number of patches. So, we obtain as follows:
	\begin{align*}
		\mbox{setup costs:} & \quad \mathcal{O}(Np+K p^6)\\
		\mbox{application costs:} & \quad \mathcal{O}(Np+K p^4) ,
	\end{align*}
	where $N=Kn^2$ is the number of unknowns, $K$ is the number of patches and $p$ is the spline degree.

	We obtain for $p\le n$ that the smoother
	is asymptotically not more expensive than the computation of the residual. The remaining parts
	of the multigrid solver (restriction, prolongation, solving on the coarsest
	grid) can also be done in optimal time, cf.~\cite[Section~5.4]{HT:2016}.
\end{remark}

As we can prove convergence only if $\mathcal{O}(p)$ smoothing steps are applied, this does
not show that the overall method has optimal complexity. However, in
Section~\ref{sec:num}, we will see that the method works well for fixed $\nu$, so in practice
the method seems to be optimal.
In the next section, we construct a multigrid solver where we can prove optimal complexity.

\subsection{An optimal variant of the additive smoother} \label{subsec:variant}

First note that the smoother $L_T$ is a robust preconditioner for $A_T + h^{-2} M_T$.
\begin{theorem}\label{thrm:equiv}
	For  all grid sizes $h$ and spline degrees $p\in\mathbb{N}$,  we obtain the relation
	$L_T \eqsim  A_T + h^{-2} M_T$  for all $T\in \mathbb T$.
\end{theorem}
\begin{proof}
	First note that Lemma~\ref{lem:equiv} states $	A_T \lesssim L_T \lesssim  A_T + h^{-2} M_T$. So, it remains
	to show that
	\begin{equation}\label{eq:thrm:equiv:whattoshow}
		h^{-2} M_T \lesssim L_T \qquad \mbox{holds for all } T\in \mathbb T.
	\end{equation}

	For $T\in  \mathbb V$, observe that from~\eqref{eq:lem:stab:v:norms:2} and \eqref{eq:lem:stab:v:norms:2a},
	it follows that $p^2 h^{-2} M_{k,T} \eqsim A_{k,T}$. By summing up, we obtain
	$p^2 h^{-2} M_{T} \eqsim A_{T}$, which shows \eqref{eq:thrm:equiv:whattoshow} as $L_T=A_T$ and $p\ge1$.

	For $T\in  \mathbb E$, observe that the combination of~\eqref{eq:lem:stab:e:res:1} and~\eqref{eq:lem:stab:e:res:2}
	yields
	\begin{align*}
		\|P_T^{\top} \ul{u}_h \|_{A_{k,T}} & \gtrsim
			\| \widehat{u}_h - \widehat{u}_h(0,0) \widehat{B}_{p,h}^{(1)} - \widehat{u}_h(0,1) \widehat{B}_{p,h}^{(m)}  \|_{L_2(\widehat{\Gamma})}^2 |\widehat{B}_{p,h}^{(1)} |_{H^1(0,1)}^2\\
			& \eqsim |\widehat{B}_{p,h}^{(1)} |_{H^1(0,1)}^2\|\widehat{B}_{p,h}^{(1)} \|_{L_2(0,1)}^{-2} \|P_T^{\top} \ul{u}_h \|_{M_{k,T}}.
	\end{align*}
	Again, using~\eqref{eq:lem:stab:v:norms:2a}, we obtain $p^2 h^{-2} M_{k,T} \lesssim A_{k,T}$ and by summing up,
	we obtain $p^2 h^{-2} M_{T} \eqsim A_{T}$, which shows \eqref{eq:thrm:equiv:whattoshow} as $L_T=A_T$ and $p\ge1$.

	For $T\in  \mathbb K$, the proof follows an idea by C.~Hofreither~\cite{CLEM}.
	Note that, in~\cite[Section~4.2]{HT:2016}, we have constructed the smoother $L_T$ on subspaces of the spline space
	obtained by a stable splitting of the whole spline space $S=S_{p,h}$ (for the particular patch) into subspaces $S_{\alpha}$.
	In two dimensions, we have defined $\sigma:=12 h^{-2}$ and
	\begin{align*}
		L_{00} & = (1+2\sigma) M_0\otimes M_0, &\qquad
		L_{01} & = M_0\otimes ((1+\sigma)M_1+K_1), \\
		L_{10} & =  ((1+\sigma)M_1+K_1)\otimes M_0 ,&\qquad
		L_{11} & = M_1\otimes M_1 + K_1\otimes M_1 + M_1\otimes K_1,
	\end{align*}
	where $M_0$, $M_1$, $K_0$ and $K_1$ are the univariate mass and stiffness matrices corresponding to the
	spaces $S_0$ and $S_1$.
	Obviously, we have $L_{00}  \ge h^{-2} M_0\otimes M_0$, $L_{10}  \ge h^{-2} M_1\otimes M_0$, and
	$L_{01}  \ge h^{-2} M_0\otimes M_1$.

	It remains to show that $L_{11} \ge h^{-2} M_1\otimes M_1$ also holds. Note that
	\cite[Theorem~3]{HT:2016} states $\|(I-Q_0)u\|_{L_2(0,1)}^2 \lesssim h^2 |u|_{H^1(0,1)}^2$, which
	yields also $\|(I-Q_0)u\|_{L_2(0,1)}^2 \lesssim h^2 |(I-Q_0)u|_{H^1(0,1)}^2$ and moreover
	\[
		h^{-2} (I-Q_0)^{\top} M (I-Q_0) \lesssim (I-Q_0)^{\top} K (I-Q_0),
	\]
	where $M$ and $K$ are the univariate mass and stiffness matrices corresponding to the whole spline space $S$.
	Note that in~\cite[Section~3.2]{HT:2016}, we
	have defined $M_1 = (I-Q_0)^\top M (I-Q_0)$
	and $K_1 = (I-Q_0)^\top K (I-Q_0)$, so we have
	$
			h^{-2} M_1 \lesssim K_1.
	$
	Using the definition of $L_{11}$, we obtain
	$L_{11} \ge  M_1\otimes K_1 \gtrsim h^{-2} M_1\otimes M_1$.

	Now, we have shown
	\[ L_{\alpha\beta}  \ge h^{-2} M_\alpha\otimes M_\beta
		\qquad \mbox{for }\alpha,\beta\in\{0,1\}.
	\]
	Using this, the fact that the spaces $S_{\alpha}$ are by construction $L_2$-orthogonal, we immediately 
	obtain~$\widehat M_T \lesssim L_T$. (Note that
	this is completely analogous to \cite[Lem.~8 and 9]{HT:2016}). Using Lemma~\ref{lem:geoequiv},
	we obtain~\eqref{eq:thrm:equiv:whattoshow}.
\qed\end{proof}

\begin{corollary}\label{corr:equiv}
	For  all grid sizes $h$ and spline degrees $p\in\mathbb{N}$,  we obtain
	\[
		A_h + h^{-2} M_h \lesssim L_h \lesssim p( A_h + h^{-2} M_h).
	\]
\end{corollary}
\begin{proof}
	We can show
	\[
		A_h + h^{-2} M_h \lesssim \sum_{T\in\mathbb T} P_T (A_T + h^{-2} M_T) P_T^{\top}
	\]
	analogously to the proof of Lemma~\ref{lem:first}.
	Using this, Theorem~\ref{thrm:equiv} and the definition of $L_h$, we obtain $A_h + h^{-2} M_h \lesssim L_h $. 
	Lemma~\ref{lem:last} states $L_h \lesssim p( A_h + h^{-2} M_h)$.
\qed\end{proof}

Based on these results, we can construct a smoother that can be applied with optimal complexity
and which yields  provably robust convergence rates.

The smoother is given by
\begin{align*}
	\widetilde{L}_h^{-1} &:= \varrho^{-1} \left(I - \left(I - \varrho L_h^{-1} (\widehat{A}_h + h^{-2} \widehat{M}_h
				)\right)^p\right) (\widehat{A}_h + h^{-2} \widehat{M}_h)^{-1} \\
			& =\left( \sum_{i=1}^p {p \choose i}
				\left( -\varrho L_h^{-1} (\widehat{A}_h + h^{-2} \widehat{M}_h)\right)^{i-1} \right) L_h^{-1},
\end{align*}
where $\varrho>0$ is chosen independent of the grid size $h$, the spline degree $p$ and the
number of patches $K$ such
that $\varrho (\widehat{A}_h + h^{-2} \widehat{M}_h) \le  L_h$. This is possible due to Corollary~\ref{corr:equiv}.
Note that $\widetilde{L}_h$ represents nothing but $p$ steps of a preconditioned Richardson method; so the
smoothing step~\eqref{eq:sm} is to be realized by
\begin{align*}
                   \ul{r}_h^{(0,\mu)} &:= \ul{f}_h-A_h\;\ul{u}_h^{(0,\mu-1)} \\
                   \ul{p}_h^{(0,\mu,1)} &:= \varrho L_h^{-1} \ul{r}_h^{(0,\mu)} \\
                   \ul{p}_h^{(0,\mu,i)} &:= \ul{p}_h^{(0,\mu,i-1)} + \varrho L_h^{-1}
                                    \left(\ul{r}_h^{(0,\mu)}-(\widehat{A}_h + h^{-2} \widehat{M}_h)\ul{p}_h^{(0,\mu,i-1)}\right) \qquad i=2,\ldots,p\\
                   \ul{u}_h^{(0,\mu)} &:= \ul{u}_h^{(0,\mu-1)} + \tau \varrho^{-1} \ul{p}_h^{(0,\mu,p)}.
\end{align*}

First observe that this method can be realized with optimal complexity.
\begin{remark}
	For applying the preconditioned Richardson method, we need (besides simple vector manipulations that can be provided with
	a complexity of $\mathcal{O}(N)$)
	to apply the smoother $L_h$ and to apply the matrix $\widehat{A}_h + h^{-2} \widehat{M}_h$. The latter can be done by
	applying it patch-wise, i.e., by computing
	\[
		\sum_{k=1}^K \left((\widehat{A}_k + h^{-2} \widehat{M}_k)  \ul{p}_h\right).
	\]
	Note that $\widehat{A}_k$ and $\widehat{M}_k$, stiffness and mass matrix on the parameter domain,
	have tensor product structure. So, multiplication with them can be
	realized with a computational complexity of $\mathcal{O}(N p)$, which is not more than the application costs of
	$L_h^{-1}$, cf. Remark~\ref{rem:1}.

	The whole smoother $\widetilde{L}_h$ consists of $p$ steps, so we have to multiply the application costs
	with $p$ and obtain:
	\begin{align*}
		\mbox{setup costs:} & \quad \mathcal{O}(Np+K p^6)\\
		\mbox{application costs:} & \quad \mathcal{O}(Np^2+K p^5).
	\end{align*}

	\newcommand{\N}{\mathfrak{N}}	\newcommand{\mm}{\mathfrak{m}}
	In a multigrid setting,
	assuming $\mathcal O (\log \mm )$ levels, where each patch has $m= \mm,  \tfrac \mm2, \tfrac \mm4, \tfrac \mm8 \ldots$
	intervals in each dimension, we obtain by adding up the overall costs for smoothing:
	\begin{align*}
	\mbox{in the V-cycle:} & \quad
	    \mathcal O (\N p^2 + K (\log \mm ) p^{6} ),\\
	\mbox{in the W-cycle:} & \quad
	    \mathcal O (\N p^2 + K \mm p^5 + K (\log \mm ) p^{6}),
	\end{align*}
	where $\N \eqsim K \mm^2 $ is the number of unknowns on the finest grid.
	The full complexity including the costs for the exact coarse-grid solver and the
	intergrid transfers is asymptotically the same.

	Under mild assumptions on the relation between $p$ and $\N$, the overall complexity is
	asymptotically not more than $\mathcal O(\N p^2)$, which is the cost for one application of the stiffness matrix.
	This shows that the multigrid cycle has optimal complexity.
\end{remark}

Now, we show that this approach leads to optimal convergence. 
\begin{lemma}\label{lem:very:last}
	For all grid sizes $h$ and spline degrees $p\in\mathbb{N}$, we have
	\begin{equation}\label{eq:lem:very:last}
		\widetilde{L}_h \eqsim A_h + h^{-2} M_h
		\qquad \mbox{and}\qquad
		L_h \le \widetilde{L}_h.
	\end{equation}
\end{lemma}
\begin{proof}
	Define $X_h:= \varrho L_h^{-1} (\widehat{A}_h + h^{-2}\widehat{M}_h)$ and
	note that $\varrho$ is chosen such that $X_h \le I$.
	Corollary~\ref{corr:equiv} states there is a constant $C$ such that
	$X_h \ge C^{-1}p^{-1} I$.
	So, we obtain $0 \le I-X_h \le (1-C^{-1}p^{-1})I $ and
	\[
		\rho ((I - X_h)^p)
			\le \left(1- \frac{1}{C p}\right)^p \le \mathbf{e}^{-1/C} < 1,
	\]
	where $\mathbf{e}$ is the Eulerian number.
	This implies 
	$
		I\eqsim
		I-(I-X_h)^p =
		\varrho \widetilde{L}_h^{-1}  (\widehat A_h + h^{-2} \widehat M_h),
	$
	and
	$\widetilde{L}_h \eqsim \widehat A_h + h^{-2} \widehat M_h$, and using
	Lemma~\ref{lem:geoequiv}
	finally the first relation in~\eqref{eq:lem:very:last}.
	
	As $0 \le I-X_h \le I$, we obtain $(I-X_h)^p\le I-X_h$, and consequently
	$ X_h \le I-(I-X_h)^p $, which implies $L_h \le \widetilde{L}_h$, the second relation in~\eqref{eq:lem:very:last}.
\qed\end{proof}

Using this Lemma and Theorem~\ref{thrm:conv}, we obtain the following theorem.
\begin{theorem}
	There are constants $c_1$ and $c_2$ that do not depend on the grid size $h$, the
	spline degree $p$, and the number of
	patches $K$ (but may depend on $C_G$, $C_N$, or $C_R$) such that 
	\begin{equation}\label{eq:thrm:final2}
			\tau L_h^{-1} A_h \le I\qquad \mbox{and}\qquad \varrho L_h^{-1} (\widehat{A}_h + h^{-2} \widehat{M}_h) \le I 
	\end{equation}
	for all $\tau \in (0,c_1]$ and all $\varrho \in (0,c_2]$. 
	For any fixed choice of $\tau$ and $\varrho$
	satisfying~\eqref{eq:thrm:final2}, there is some $\nu_0$ that does not depend on $p$, $h$, or $K$
	such that the proposed two-grid method converges for
	any choice of the number of smoothing steps $\nu > \nu_0$
	with a  convergence rate $q = \nu_0/\nu$, i.e.,
	\[
		\| \ul{u}_h^{(1)} - A_h^{-1} \ul{f}_h \|_{A_h + h^{-2} M_h} \le \frac{\nu_0}{\nu}
			\| \ul{u}_h^{(0)} - A_h^{-1} \ul{f}_h \|_{A_h + h^{-2} M_h}.
	\]
\end{theorem}
\begin{proof}
	Note that $\tau L_h^{-1} A_h\le I$ and Lemma~\ref{lem:very:last} imply that $\tau A_h\le \widetilde{L}_h$.
	Using this and Lemma~\ref{lem:very:last}, we obtain the conditions of
	Theorem~\ref{thrm:conv}, which yields the desired statement.
\qed\end{proof}
Due to~\cite[Theorem~4]{HTZ:2016}, we know that also the W-cycle multigrid method converges.

\section{Numerical experiments}\label{sec:num}

In this section, we present numerical experiments that illustrate the efficiency of
the proposed multigrid solver.
The multigrid solver was implemented in C++ based on the G+Smo
library~\cite{gismoweb}.

\subsection{The unit square}

In this section, we consider the domain $\Omega=(-0.6,1.4)^2$, which is decomposed into four
patches $\Omega_1 = (-0.6,0.4)^2$, $\Omega_2=(0.4,1.4)\times (-0.6,0.4)$, $\Omega_3=(-0.6,0.4)\times(0.4,1.4)$,
and  $\Omega_4=(0.4,1.4)^2$; in all cases the geometry transformation is just
a translation. We solve the problem
\begin{equation}\label{eq:num:prob}
	\begin{aligned}
		-\Delta u = 2\pi^2 \, \sin(\pi x) \,  \sin(\pi y)& \qquad \mbox{in } \Omega \\
		u = g:= \sin(\pi x) \,  \sin(\pi y)& \qquad \mbox{on }\partial \Omega
	\end{aligned}
\end{equation}
and note that $g$ is the exact solution of the problem. On the coarsest
grid level $\ell=0$, the whole patch is just one element. The grid levels $\ell=1,2,\ldots,$
are obtained by uniform refinement. The coarsest grid which is actually used in the
multigrid method is chosen such that for all patches the condition $m>p$
holds, i.e., that the number of intervals is more than $p$, cf.~\cite[Section~6.1]{HT:2016}.

\begin{table}[ht]
\begin{center}
    \begin{tabular}{l|rrrrrrr}
    \toprule
    $\ell\,\backslash\, p$    &  2 &  3 &  4 &  5 &  6 &  7 &  8  \\
    \midrule
	 4                 & 39 & 32 & 22 & 24 & 21 & 20 & 21  \\
	 5                 & 56 & 40 & 32 & 28 & 28 & 32 & 33  \\
	 6                 & 60 & 44 & 37 & 31 & 31 & 34 & 37  \\
	 7                 & 61 & 45 & 37 & 32 & 31 & 35 & 37  \\
	 8                 & 63 & 45 & 38 & 32 & 31 & 35 & 37  \\
	\bottomrule
	\end{tabular}
\end{center}
\caption{Multigrid for the unit square with $1+1$ steps of smoother $L_h$ as iterative method}
\label{tab:S:1}
\end{table}

\begin{table}[ht]
\begin{center}
    \begin{tabular}{l|rrrrrrr}
    \toprule
    $\ell\,\backslash\, p$    &  2 &  3 &  4 &  5 &  6 &  7 &  8  \\
    \midrule
	 4                 & 14 & 12 & 11 & 11 & 10 & 11 & 10  \\
	 5                 & 16 & 15 & 14 & 14 & 13 & 13 & 12  \\
	 6                 & 18 & 16 & 15 & 15 & 14 & 14 & 14  \\
	 7                 & 18 & 16 & 16 & 15 & 14 & 14 & 14  \\
	 8                 & 19 & 16 & 16 & 15 & 15 & 15 & 14  \\
	\bottomrule
	\end{tabular}
\end{center}
\caption{Multigrid for the unit square with $1+1$ steps of smoother $L_h$ as preconditioner for conjugate gradient}
\label{tab:S:2}
\end{table}
\begin{table}[ht]
\begin{center}
    \begin{tabular}{l|rrrrrrr}
    \toprule
    $\ell\,\backslash\, p$    &  2 &  3 &  4 &  5 &  6 &  7 &  8  \\
    \midrule
	 4                 & 29 & 11 & 8 & 7 & 6 & 5 & 5  \\
	 5                 & 48 & 13 & 10 & 8 & 7 & 7 &6  \\
	 6                 & 55 & 14 & 12 & 9 & 8 & 7 & 7  \\
	 7                 & 56 & 14 & 12 & 9 & 8 & 8 & 7  \\
	 8                 & 59 & 15 & 13 & 9 & 8 & 8 & 7  \\
	\bottomrule
	\end{tabular}
\end{center}
\caption{Multigrid for the unit square with $1+1$ steps of smoother $\widetilde{L}_h$ as iterative method}
\label{tab:S:3}
\end{table}
As first numerical example, we set up the W-cycle multigrid method with the proposed
smoother $L_h$ (cf. Section~\ref{sec:additive}), where $1$ pre- and $1$ post-smoothing
step is applied. As damping parameter, we choose $\tau = 0.95$. The parameter
in the subspace-corrected mass smoother, cf.~\cite{HT:2016}, is chosen as
$\sigma := \tfrac{1}{0.2} h^{-2}$. The iteration counts required to reduce the initial
error by a factor of $\epsilon=10^{-8}$ are given in Table~\ref{tab:S:1}. We observe that
the method shows robustness both in the grid size $h_{\ell}:=2^{-\ell}$ (which was proven) and the spline
degree $p$ (where this is only proven for $\mathcal{O}(p)$ smoothing steps), where we observe --
as in~\cite{HT:2016} -- that the convergence gets slightly better if $p$ is increased.
We observe that, as expected, the iteration counts are improved if we use the multigrid
method as a preconditioner for a conjugate gradient method, cf. Table~\ref{tab:S:2}.
Similar iteration numbers are obtained for the V-cycle.

Finally, in Table~\ref{tab:S:3}, we consider the results for the smoother $\widetilde{L}_h$ (cf. Section~\ref{subsec:variant}).
Here, we choose $\sigma$ as above, $\varrho = 0.95$ and $\tau = 1$. Again $1+1$ smoothing steps
are applied in a W-cycle multigrid iteration.
We observe again that the method shows  robustness in the grid size and the spline degree (which was
proven). We observe 
that the iteration numbers decrease if the spline degree is increased.
For large spline degrees $p$ the iteration numbers are significantly smaller than for the smoother $L_h$,
however the numerical experiments seem to indicate that effect does not justify
the additional effort required to realize the smoother $\widetilde{L}_h$.

\subsection{The L-shaped domain}

In this section we consider the first non-trivial example.
We extend the method beyond the case covered by the convergence theory
to the L-shaped domain
\[
	\Omega = \{(x,y)\in (-0.6,1.4)^2\;:\; x<0.4 \vee y<0.4\},
\]
where the regularity assumption does not hold due to the reentrant corner.
The domain is 
decomposed into three
patches $\Omega_1 = (-0.6,0.4)^2$, $\Omega_2=(0.4,1.4)\times (-0.6,0.4)$, and
$\Omega_3=(-0.6,0.4)\times(0.4,1.4)$; in all cases the geometry transformation is just a translation. 
Again, we solve for the problem~\eqref{eq:num:prob}.

\begin{table}[ht]
\begin{center}
    \begin{tabular}{l|rrrrrrr}
    \toprule
    $\ell\,\backslash\, p$    &  2 &  3 &  4 &  5 &  6 &  7 &  8  \\
    \midrule
	 4                 & 37 & 33 & 22 & 24 & 18 & 21 & 19  \\
	 5                 & 56 & 39 & 32 & 28 & 26 & 31 & 31 \\
	 6                 & 60 & 44 & 37 & 31 & 29 & 34 & 35  \\
	 7                 & 61 & 45 & 37 & 32 & 31 & 35 & 37  \\
	 8                 & 63 & 45 & 38 & 32 & 31 & 35 & 35  \\
	\bottomrule
	\end{tabular}
\end{center}
\caption{Multigrid for the L-shaped domain with $1+1$ steps of smoother $L_h$ as iterative method}
\label{tab:L:1}
\end{table}

\begin{table}[ht]
\begin{center}
    \begin{tabular}{l|rrrrrrr}
    \toprule
    $\ell\,\backslash\, p$    &  2 &  3 &  4 &  5 &  6 &  7 &  8  \\
    \midrule
	 4                 & 13 & 12 & 11 & 11 & 10 & 11 & 10  \\
	 5                 & 16 & 15 & 14 & 14 & 13 & 13 & 12  \\
	 6                 & 18 & 16 & 15 & 15 & 14 & 14 & 13  \\
	 7                 & 18 & 16 & 16 & 15 & 15 & 14 & 14  \\
	 8                 & 18 & 16 & 16 & 15 & 15 & 15 & 14  \\
	\bottomrule
	\end{tabular}
\end{center}
\caption{Multigrid for the L-shaped domain with $1+1$ steps of smoother $L_h$ as preconditioner for conjugate gradient}
\label{tab:L:2}
\end{table}
Again, we set up the W-cycle multigrid method with 1+1 smoothing steps of the proposed
smoother $L_h$.
We choose $\tau = 0.95$ and $\sigma = \tfrac{1}{0.2} h^{-2}$. The iteration counts required to
reduce the initial error by a factor of $\epsilon=10^{-8}$ are given in Table~\ref{tab:L:1}. We observe that
the iteration counts are similar to those for the unit square and that the
method shows again robustness in the grid size and the spline degree.
We observe that, as expected, the iteration counts are improved if we use the multigrid
method as a preconditioner for a conjugate gradient method, cf. Table~\ref{tab:L:2}.

\subsection{The Yeti footprint}
As third domain, we consider the Yeti footprint, cf. Figure~\ref{fig:1}. This domain is a popular model problem
for the IETI method \cite{KleissEtAl:2012}. This domain has non-trivial
geometry transformation functions. 
\begin{figure}[h]
\begin{center}
	\includegraphics[width=.22\textwidth]{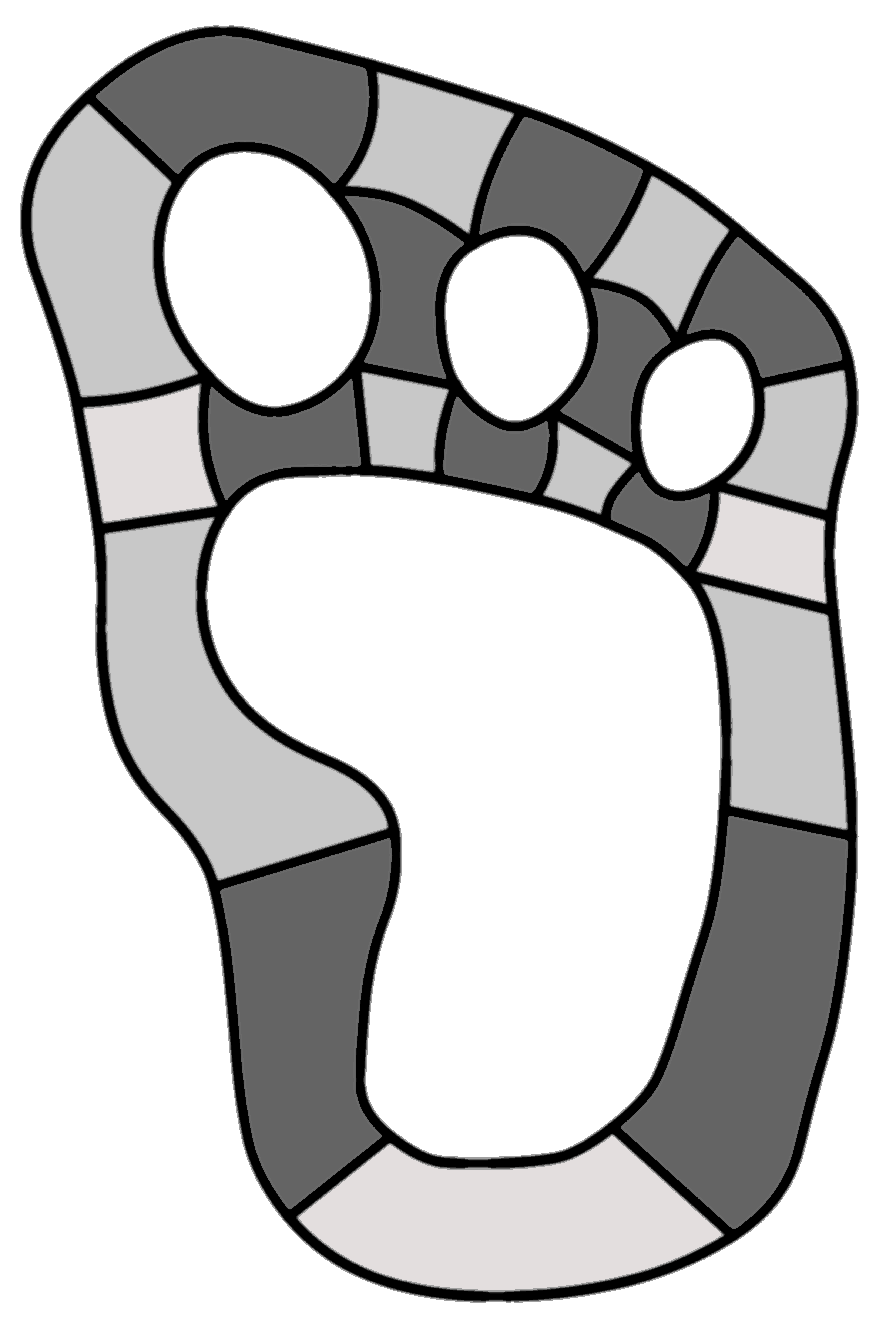}
\end{center}
\caption{The Yeti footprint}
\label{fig:1}
\end{figure}

As the domain has a smooth boundary, it is covered by the theory
presented within the paper. The domain is decomposed into $21$ patches, which can be seen
in Figure~\ref{fig:1}. Again, we solve for the problem~\eqref{eq:num:prob}.
For this example, we have to reduce the damping parameter.
We choose $\tau = 0.25$ and $\sigma = \tfrac{1}{0.2} h^{-2}$. 
If the multigrid method is used as an iterative scheme, the method suffers from the geometry
transformation, so robust convergence is only obtained for 2+2 smoothing steps, cf. Table~\ref{tab:I:1}.
If the method is used as a preconditioner for a conjugate gradient method, again
1+1 smoothing steps are sufficient for rather good convergence rates, cf.
Table~\ref{tab:I:3}. 
Again we observe robustness both in the grid size and the spline degree.
Similar iteration counts are obtained for the V-cycle.

\begin{table}[h]
\begin{center}
    \begin{tabular}{l|rrrrrrr}
    \toprule
    $\ell\,\backslash\, p$    &  2 &  3 &  4 &  5 &  6 &  7 &  8  \\
    \midrule
	 4                & 182 & 194 & 176 & 172 & 182 & 153 & 160 \\
	\bottomrule
	\end{tabular}
\end{center}
\caption{Multigrid for the Yeti footprint with $2+2$ steps of smoother $L_h$ as iterative method}
\label{tab:I:1}
\end{table}

\begin{table}[h]
\begin{center}
    \begin{tabular}{l|rrrrrrr}
    \toprule
    $\ell\,\backslash\, p$    &  2 &  3 &  4 &  5 &  6 &  7 &  8  \\
    \midrule
	 4                 & 46 & 46 & 45 & 43 & 43 & 40 & 41  \\
	 5                 & 48 & 47 & 47 & 46 & 45 & 45 & 44  \\
	 6                 & 48 & 49 & 48 & 47 & 47 & 46 & 45  \\
   7                 & 49 & 50 & 49 & 49 & 48 & 47 & 47 \\
	\bottomrule
	\end{tabular}

\end{center}
\caption{Multigrid for the Yeti footprint with $1+1$ steps of smoother $L_h$ as preconditioner for conjugate gradient}
\label{tab:I:3}
\end{table}

\begin{table}[h]
\begin{center}
    \begin{tabular}{ll|rrrr}
    \toprule
    $\ell \,\backslash\, p$    && 2 & 4 & 6 & 8\\
    \midrule
	 4                 & \# of unknowns        & 26\,368    &29\,444   & 32\,688     & 36\,100  \\
	                   & W-cycle    &      2.5 s &   4.2 s  &       9.8 s & 17.6 s \\
	                   & V-cycle    &      1.8 s &    3.1 s &       8.0 s & 15.0 s \\
\midrule
	 5                 & \# of unknowns        & 103\,936   & 109\,956 &116\,144      &122\,500 \\
	                   & W-cycle    &   10 s     &   17 s   &         30 s & 48 s \\
	                   & V-cycle    &        7 s &     12 s &         21 s & 35 s \\
\midrule
	 6                 & \# of unknowns        & 412\,672  &  424\,580 &436\,656      & 448\,900 \\
	                   & W-cycle    &      43 s &      66 s &       106 s  & 156 s \\
	                   & V-cycle    &      30 s &      47 s &        74 s  & 112 s \\
\midrule
	 7                 & \# of unknowns        & 1\,644\,544  &  1\,668\,228 &1\,693\,080      & 1\,716\,100  \\
	                   & W-cycle    &    185 s &      284 s &      465 s  & 712 s \\
	                   & V-cycle    &     115 s &      187 s &       299 s  & 511 s \\
	\bottomrule
	\end{tabular}

\end{center}
\caption{Multigrid for the Yeti footprint with $1+1$ steps of smoother $L_h$ as preconditioner for conjugate gradient}
\label{tab:I:X}
\end{table}

In Table~\ref{tab:I:X}, we show actual CPU times required for to execute the numerical tests from
Table~\ref{tab:I:3} on a standard personal computer\footnote{12 core
Intel(R) Xeon(R) CPU, 3.20GHz with 15.6 GiB RAM} without any parallelization.
The CPU times include the setup of the multigrid solver
and the solution of the problem (but it excludes the assembling of the stiffness matrix). We observe that
for $h$-refinement, the CPU times grow linearly with the number of unknowns.
For the spline degree, we observe that the complexity grows less than quadratically
with the spline degree. 
Concluding, we observe that the overall complexity does not exceed $\mathcal{O}(Np^2)$,
the number of non-zero entries of the stiffness matrix.

\section{Conclusions}\label{sec:conclusions}

We have introduced a multigrid smoother based on an additive domain decomposition approach and have proven that
its convergence rates are robust both in the grid size and the spline degree. The proof only holds if $\mathcal{O}(p)$ smoothing
steps are applied, the experiments show however that $1+1$ smoothing steps are enough. So, following the
numerical experiments, the proposed smoother yields an optimal multigrid method.

Moreover, we have given a variant of the smoother in Section~\ref{subsec:variant}, where we could actually
prove optimal complexity.
The numerical experiments seem to indicate that the original smoother is always superior to that variant, so it is more
of theoretical interest.

\section*{Acknowledgments} The author thanks C. Hofer and C. Hofreither for fruitful discussions on topics
related to this publication.

\bibliographystyle{amsplain}
\bibliography{references}

\end{document}